\documentclass{amsart}
\usepackage{bbm}
\usepackage{dsfont}
\usepackage{relsize}
\usepackage{amsthm}
\usepackage[breaklinks]{hyperref}
\usepackage{blindtext}
\usepackage{amsmath,amssymb}
\usepackage{mathtools}
\usepackage{autonum}
\usepackage{enumitem}
\usepackage[usenames,dvipsnames,svgnames]{xcolor}
\usepackage{mathrsfs}
\usepackage{csquotes}
\usepackage[]{hyperref}
\hypersetup{colorlinks=true,linkcolor=WildStrawberry,citecolor=cyan}
\newtheorem{thm}{Theorem}[section]

\newtheorem{lem}[thm]{Lemma}

\theoremstyle{definition}
\newtheorem{defn}[thm]{Definition}
\theoremstyle{remark}
\newtheorem{rem}[thm]{Remark}

\numberwithin{equation}{section}

\newcommand{\To}{\longrightarrow}

\makeatletter
\renewcommand\paragraph{\@startsection{paragraph}{4}{\z@}%
	{-2.5ex\@plus -1ex \@minus -.25ex}%
	{1.25ex \@plus .25ex}%
	{\normalfont\normalsize\centering\bfseries}}
\newcommand{\nocontentsline}[3]{}
\let\origcontentsline\addcontentsline
\newcommand\stoptoc{\let\addcontentsline\nocontentsline}
\newcommand\resumetoc{\let\addcontentsline\origcontentsline}
\makeatother
\begin{document}
	\title
	[Fujita exponent for fractional sub-Laplacian in the Heisenberg group]
	{Fujita exponent for the fractional sub-Laplace semilinear heat equation with forcing term on the Heisenberg group}
	
	\author[P.\,Oza and D.\,Suragan]
	{Priyank Oza and Durvudkhan Suragan}
	\address{Priyank\,Oza \hfill\break
		Department of Mathametics\\
		Nazarbayev University \newline
		53 Kabanbay Batyr Ave, 010000 Astana, Kazakhstan.}
	\email{priyank.kumar@nu.edu.kz, priyank.oza3@gmail.com}
	\address{Durvudkhan\,Suragan \hfill\break
		Department of Mathametics\\
		Nazarbayev University \newline
		53 Kabanbay Batyr Ave, 010000 Astana, Kazakhstan.}
	\email{durvudkhan.suragan@nu.edu.kz}
	\thanks{Submitted \today.  Published-----.}
	\begin{abstract}
		We study the following semilinear heat equation in the Heisenberg group $\mathbbm{H}^N$:  
		\begin{align}	
			\begin{cases}
				\partial_tu+(-\Delta_{\mathbbm{H}^N})^su=|u|^p+f &\text{in } \,\mathbbm{H}^N\times (0,T),\\
				u(0,\cdot)=u_0(\cdot) &\text{in } \,\mathbbm{H}^N,
			\end{cases}	
		\end{align}
		where $(-\Delta_{\mathbbm{H}^N})^s$ denotes the fractional sub-Laplacian of order $s\in (0,1)$ on $\mathbbm{H}^N$. We establish that the Fujita exponent, a critical threshold that delimits different dynamical regimes of this equation, is $$p_F\coloneqq\frac{Q}{Q-2s},$$ where $Q\coloneqq 2N+2$ is the homogeneous dimension of $\mathbbm{H}^N$. We prove the existence of global-in-time solutions for the supercritical case $(p>p_F),$ and the non-existence of global-in-time solutions for the subcritical case $(1<p<p_F).$ For the critical case $p=p_F,$ we provide a class of functions for which $u$ blows up in finite time. These results extend the classical Fujita phenomenon to a sub-Riemannian setting with the nonlocal effects of the fractional sub-Laplacian. Our proof methods intertwine analytic techniques with the geometric structure of the Heisenberg group.
	\end{abstract}	
	
	\maketitle
	{\hypersetup{linkcolor=black}
		\tableofcontents}
	\section{Introduction}
	The study of the \textit{Fujita exponent} has gained significant attention over the years. This exponent represents a critical value that characterizes the time evolution of solutions to corresponding semilinear heat equations. The seminal work in this area was carried out by Fujita \cite{Fujita}, who investigated the following problem:
	\begin{align}\label{Fu}
		\begin{cases}
			\partial_tu-\Delta u=u^p &\text{ in }\mathbb{R}^N\times (0,T),\\
			u(\cdot,0)=u_0(\cdot) &\text{ in }\mathbb{R}^N,
		\end{cases}
	\end{align}
	for $p>1,$ $u_0\geq 0,$ and $T,$ which represents the maximal time of existence of the solution. Fujita proved the following results:
	\begin{enumerate}
		\item [(i)] There exists a global-in-time solution to \eqref{Fu} for $p> \frac{N+2}{N}.$
		\item [(ii)] All solutions to \eqref{Fu} blow up in finite time for $p\in\left(1,\frac{N+2}{N}\right).$
	\end{enumerate}
	Here $$\frac{N+2}{N}$$ is known as the Fujita exponent for problem \eqref{Fu}. Furthermore, Hayakawa \cite{Haya}, Sugitani \cite{Sugitani} and Kobayashi et al. \cite{Koba} addressed the case $p=\frac{N+2}{N}$, showing that there is no global solution in this case. Subsequently, Bandle et al. \cite{Bandle} examined the following semilinear heat equation with a forcing term: 
	\begin{align}\label{Band}
		\begin{cases}
			\partial_tu-\Delta u=|u|^p+f &\text{ in }\mathbb{R}^N\times (0,T),\\
			u(\cdot,0)=u_0(\cdot) &\text{ in }\mathbb{R}^N,
		\end{cases}
	\end{align}
	where $\int_{\mathbb{R}^N} f>0,$
	presenting a non-homogeneous version of \eqref{Fu}. They showed that the Fujita exponent for \eqref{Band} is given by $\frac{N}{N-2},\; N>2.$ 
	
	Following this, numerous studies have explored the Fujita phenomenon involving both local and nonlocal operators in the Euclidean setting. Although we highlight a few closely related works, this list is by no means exhaustive, as many other studies have made significant contributions to the field.
	For related results on local diffusion operators such as $p$-Laplacian and porous medium equations, readers are referred to the monographs \cite{Quit, Samar} and the survey by Levine \cite{Levine}. Further studies on the decay rate of solutions to equations of the form \eqref{Fu} can be found in \cite{Fila, Fila2, Fila3} and the references therein. Jleli et al. \cite{Jleli} generalized the problem \eqref{Fu} to an equation with forcing term of the form $t^\sigma f(x)$ for some constant $\sigma\in (-1,\infty)\setminus\{0\}.$ In this case, the critical exponent depends explicitly on $\sigma.$ Analogous results for the semilinear fractional heat equation in Euclidean space can be found in \cite{Guedda, Majd, Nagasa, XWang},  while \cite{Biagi, Del Pezzo} address on classes of mixed operators.
	
	Over the past decades, the study of partial differential equations in non-Euclidean settings has garnered significant attention, as it bridges analysis with geometry and mathematical physics.  In particular, partial differential equations in the Heisenberg group arise naturally in fields such as the study of magnetic trajectories on nilmanifolds \cite{ferro}, image segmentation models \cite{image}, non-Markovian coupling of Brownian motions \cite{non-markov}, quantum mechanics \cite{quantum}. In these contexts, sub‑Riemannian geometry plays a crucial and intrinsic role.

	A significant amount of research has been dedicated to studying $\Delta_\infty$ (infinity Laplace) equations and degenerate equations in non-Euclidean settings, particularly on the Heisenberg group, as well as more generally in sub-Riemannian and Carnot-Carath\'eodory spaces. Notably, Bieske \cite{Bieske} explored infinite harmonic functions in the Heisenberg group using viscosity solutions, while Wang \cite{Wang} established the uniqueness of viscosity solution to $\Delta_\infty$ equation on Carnot groups. Additionally, Wang \cite{Wang 2} studied the removable singularities of viscosity subsolutions to degenerate elliptic Pucci operators in the Heisenberg group. Ferrari and Vecchi \cite{Ferrai} demonstrated the H\"older regularity of uniformly continuous and bounded viscosity solutions of fully nonlinear degenerate equations in $\mathbbm{H}^1$. For further insights into existence results and Liouville and Harnack type qualitative properties of fundamental solutions, we refer to \cite{Cutri}. Li and Wang \cite{Li} introduced a comparison principle for sub-elliptic equations in the Heisenberg setting. Furthermore, in \cite{bard1}, a strong comparison principle is established for degenerate elliptic equations, including those involving Pucci's extremal operators over H\"ormander vector fields. 
	
	Recently, Palatucci and Piccinini \cite{Palat2, Palatucci} considered a broad class of nonlinear integro-differential operators in the Heisenberg group $\mathbbm{H}^N.$ They proved general Harnack inequalities for weak solutions to the Dirichlet problem associated with these nonlinear integro-differential operators, thereby enriching the understanding of such operators in the non-Euclidean context. More recently, the first author and Tyagi \cite{Oza Heis}  introduced mixed local and nonlocal operators in the Heisenberg group and studied the existence and regularity of viscosity solutions to the corresponding equation. We also refer to the work by Branson et al. \cite{Branson} on pseudodifferential operators in $\mathbbm{H}^N$, which explores the Moser-Trudinger inequalities on Cauchy-Riemann (CR) spheres. Manfredini et al. \cite{Manfredi} investigated H\"older continuity and boundedness estimates for nonlinear fractional equations in $\mathbbm{H}^N,$ where they considered equations driven by integro-differential operators. Intereseted readers may also see \cite{Piccinini} and the references therein for more insights.

	In the context of Fujita theory, we recall some recent results in non-Euclidean settings followed by the seminal work of Pascucci \cite{Pasu}, who generalized Fujita's work to a stratified Lie group. Recently, the second author and Talwar \cite{Talwar} studied the following semilinear subelliptic heat equation with a forcing term $f$ on an arbitrary stratified Lie group $\mathbbm{G}$:
	\begin{align}
		\begin{cases}
			\partial_tu-\Delta_{\mathbbm{G}}u=|u|^p+f, &\text{ in } (0,T)\times\mathbbm{G},\\
			u(0,\cdot)=u_0, &\text{ in } \mathbbm{G}.
		\end{cases}
	\end{align}
	 The operator $\Delta_{\mathbbm{G}}$ is defined as
	\begin{align}
		\Delta_{\mathbbm{G}}u\coloneqq-\sum_{i=1}^mX_i^2,
	\end{align}
	where $\{X_i\}_{i=1}^m$ is a system of generators for $\mathbbm{G}.$ In \cite{Berik}, Borikhanov et al. considered the following semilinear heat problem driven by the Heisenberg sub-Laplace operator:
	\begin{align}\label{Ber}
		\begin{cases}
			\partial_tu-\Delta_{\mathbbm{H}^N}u=|u|^p+f &\text{ in } \mathbbm{H}^N\times (0,T),\\
			u(0,\cdot)=u_0(\cdot) &\text{ in } \mathbbm{H}^N, 
		\end{cases}	
	\end{align}
	for $p>1.$ See also \cite{Ahmad, Alsae, Amb, Poho, Ruzha} for more related earlier works.
	
	Motivated by these advancements in the study of PDEs in non-Euclidean settings, we explore the Fujita theory for the fractional sub-Laplacian on the Heisenberg group. In this paper, we show that the Fujita exponent for the problem 
	\begin{align}\label{eq 0.1}
		\begin{cases}
			\partial_tu+(-\Delta_{\mathbbm{H}^N})^su=|u|^p+f &\text{in } \,\mathbbm{H}^N\times (0,T),\\
			u(0,\cdot)=u_0(\cdot) &\text{in } \,\mathbbm{H}^N,
		\end{cases}	
	\end{align}
	is $$p_F=\frac{Q}{Q-2s}.$$
	Here, $(-\Delta_{\mathbbm{H}^N})^s$ denotes the fractional sub-Laplacian of order $s\in (0,1)$ on the Heisenberg group $\mathbbm{H}^N$ (as defined by \eqref{frac}), and $Q\coloneqq 2N+2$ is the homogeneous dimension of $\mathbbm{H}^N.$ 
	
	It is worth mentioning that the study of nonlocal operators in non-Euclidean settings is relatively recent, with only a few papers dedicated to this topic. The main objective of this work is to extend the Fujita phenomenon to these nonlocal operators.
    
	Our first main result refers to the local well-posedness of solutions to \eqref{eq 0.1}. This reads as follows:
	\begin{thm}[Local existence]\label{Local}
		Given $p>1$ and $u_0,f\in L^\infty(\mathbbm{H}^N),$ the following holds:
		\begin{enumerate}
			\item [(i)] There exists a fixed $T=T(u_0)>0,$ and a unique mild solution $u\in C([0,T],L^\infty(\mathbbm{H}^N))$ to \eqref{eq 0.1}. 
			\item [(ii)] The solution can be uniquely extended to a maximum interval $[0,T_{\text{max}}).$ Moreover, if $T_{\text{max}}$ is finite, then $\|u(t)\|_{L^\infty(\mathbbm{H}^N)}\To\infty$ as $t\To T_{\text{max}}.$
			\item [(iii)] Let $u_0,$ $f\in L^\infty(\mathbbm{H}^N)\cap L^q(\mathbbm{H}^N)$ for $q\in [1,\infty].$ Then
			\begin{align}
				u\in C([0,T_{\text{max}}),L^\infty(\mathbbm{H}^N))\cap C([0,T_{\text{max}}),L^q(\mathbbm{H}^N)).
			\end{align}
		\end{enumerate} 
	\end{thm}
	
	Subsequently, we state the next results of this paper that give the associated Fujita exponent for \eqref{eq 0.1}, which is the main goal of the present paper.
	\begin{thm}[Supercritical  Case]\label{Thm Glo2}
		For any $s\in (0,1),$ and $p>\frac{Q}{Q-2s},$ there exists $f$ and the initial data $u_0$ such that the problem \eqref{eq 0.1} admits a positive global-in-time solution.
	\end{thm}
	
	\begin{thm}[Subcritical case]\label{Thm Glo1}
		Let $s\in (0,1),$ $p\in (1,\frac{Q}{Q-2s}),$ $u_0,f\in L^\infty(\mathbbm{H}^N)\cap L^1_{\text{loc}}(\mathbbm{H}^N)$ with $u_0\geq 0$ and $\int_{\mathbbm{H}^N}f>0.$ Then \eqref{eq 0.1} does not possess a global-in-time solution.
	\end{thm}
	
	\begin{thm}[Critical case]\label{Thm Glo3}
		Let $s\in (0,1),$ $p=\frac{Q}{Q-2s},$ $u_0,f\in L^\infty(\mathbbm{H}^N)\cap L^1_{\text{loc}}(\mathbbm{H}^N)$ with $ u_0\geq 0$ and $\int_{\mathbbm{H}^N}f>0$ with $f(\xi)\geq \|\xi\|^{\alpha-Q}_{\mathbbm{H}^N},$ for $\|\xi\|_{\mathbbm{H}^N}\geq 1$ and $\alpha\in(0,Q).$ Then \eqref{eq 0.1} does not possess a global-in-time solution.
	\end{thm}

	\begin{rem}
		This work builds upon existing Fujita-type results for the sub-Laplacian in the Heisenberg group, extending them to their non-local analogues. As $s\To 1,$ the operator in \eqref{eq 0.1}  approaches the operator in \eqref{Ber} (see \cite{Corni, Ferrarii}). Therefore, the problem in \eqref{eq 0.1} can be viewed as a natural nonlocal extension of the problem in \eqref{Ber}. 
	\end{rem}
	
	Let us briefly outline our main strategy and the underlying heuristic ideas of our approach. Our first step, presented in Theorem \ref{Local}, is to establish the local well-posedness of solutions to \eqref{eq 0.1}. This result is obtained via the Banach fixed point theorem, applied in the function space:
	\begin{align}
		\Theta^s_T\coloneqq\bigg\{v\in C([0,T],L^\infty(\mathbbm{H}^N)):\|v\|_{L^\infty(\mathbbm{H}^N)}\leq 2\delta(u_0,f), v(0)=u_0\bigg\},
	\end{align}
	where $$\|v\|_{\Theta^s_T}=\|v\|_{L^\infty((0,T),L^\infty(\mathbbm{H}^N))}$$  and $$\delta(u_0,f)=\max\{\|u_0\|_{L^\infty(\mathbbm{H}^N)}, \|f\|_{L^\infty(\mathbbm{H}^N)}\}.$$
	Next, we introduce the notions of mild and weak solutions to \eqref{eq 0.1} in Definitions \ref{Mild} and \ref{Weak}, respectively. In Lemma \ref{mild to weak}, we show that any mild solution to \eqref{eq 0.1} is also a weak solution. Building upon this, we demonstrate that the corresponding Fujita exponent is $\frac{Q}{Q-2s},$ and provide the global existence of solutions for exponents $p>\frac{Q}{Q-2s}$ in Theorem \ref{Thm Glo2}.  The proof of this theorem relies on a few key lemmas, \ref{Lem Q1}--\ref{Lem Q}, which provide crucial estimates involving the Kor\'anyi distance (see \eqref{diss}). In the proof, we use the Banach fixed point theorem in the following space:
	\begin{align}
		\Theta^s_M\coloneqq\bigg\{v\in C([0,\infty),C_c(\mathbbm{H}^N)):0\leq v(t,\xi)\leq \frac{M}{1+d_{\mathbbm{H}^N}(\xi)^{Q-2s}} \text{ for }t\in [0,\infty),\, \xi\in\mathbbm{H}^N\bigg\}, \text{ for }M>0.
	\end{align}
	Subsequently, we prove in Theorem \ref{Thm Glo1} that if
	$p< \frac{Q}{Q-2s},$
	then every weak solution to \eqref{eq 0.1} blows up in finite time. Finally, in Theorem \ref{Thm Glo3}, we address the critical case $p=\frac{Q}{Q-2s},$ and prove that for a suitable class of functions $f,$ the problem \eqref{eq 0.1} does not admit any global-in-time solution.
	
	Thus, the main objective of this paper is to first establish the local existence of solutions to \eqref{eq 0.1}, as stated in Theorem \ref{Local}. We then examine the existence and non-existence of global-in-time solutions depending on the range of $p,$ which follows from Theorems \ref{Thm Glo2}, \ref{Thm Glo1} and \ref{Thm Glo3}.
	
	The organization of this paper is as follows. In Section \ref{notations}, we present the basic definitions and notations as well as introduce the framework in which we work. Section \ref{wellposedness} is devoted to the local existence result, while Section \ref{Fujitaexponent} contains the proofs of the results concerning global-in-time solutions. 
	
	Throughout the proofs of our main results, the constant $C$ may vary from line to line. However, we explicitly indicate whenever it depends on any parameters.
	\section{Preliminaries and notations}
    \label{notations}
	We begin by briefly recalling some fundamental concepts related to the Heisenberg group $\mathbbm{H}^N.$ Elements in $\mathbbm{H}^N$ are represented as
	\begin{align}
		\xi\coloneqq(z,w)=(x_1,\dots,x_N,y_1,\dots,y_N,w),	
	\end{align}
	where $x_i,\,y_i\in\mathbb{R}$ for $1\leq i\leq N,$ and $w\in\mathbb{R}.$
	The Heisenberg group $\mathbbm{H}^N$ is defined as the triplet $\big(\mathbb{R}^{2N+1},\circ,\{\Phi_\lambda\}\big)$, where $\mathbb{R}^{2N+1}$ is the underlying space, $\circ$ denotes the group law, and $\{\Phi_\lambda\}$ represents the dilation group.
	The group law $\circ$ is defined as follows for two elements $\xi=(x,y,w)$ and $\xi'=(x',y',w')\in \mathbbm{H}^N:$
	\begin{align}
		\xi\circ \xi'&=\big(x+x',y+y',w+w'+2\langle y,x'\rangle- 2\langle x,y'\rangle\big),
	\end{align}
	where $\langle\cdot,\cdot\rangle$ denotes the standard inner product in $\mathbb{R}^N.$ Expanding this for components, we have
	\begin{align}
		\xi\circ\xi'=\bigg(x_1+x_1',\dots,x_N+x_N',y_1+y_1',\dots,y_N+y_N',w+w'+2\sum_{i=1}^{N}\big(y_ix_i'-x_iy_i'\big)\bigg).
	\end{align}
	The Heisenberg group $(\mathbb{R}^{2N+1},\circ)$ is a Lie group with identity element at the origin $\textbf{0},$ and the inverse of any element $\xi=(x,y,w)$ is simply given by $\xi^{-1}=(-x,-y,-w).$  Also, the dilation group $\{\Phi_\lambda\}_{\lambda>0}$ is defined as a family of mappings
	\begin{align}
		\Phi_\lambda: \mathbb{R}^{2N+1}\longrightarrow\mathbb{R}^{2N+1},
	\end{align}
	such that
	\begin{align}
		\xi\longmapsto\Phi_\lambda(\xi)\coloneqq\big(\lambda x,\lambda y,\lambda^2 w\big).
	\end{align} 
	The group $\mathbbm{H}^N$ is also referred to as the Heisenberg-Weyl group in $\mathbb{R}^{2N+1}$. The Jacobian basis of the Heisenberg Lie algebra of $\mathbbm{H}^N$ is given by the following vector fields:
	\begin{align}
		X_i=\partial_{x_i}+2y_{i}\partial_w,\, Y_{i}=\partial _{y_i}-2x_i\partial_w, \, 1\leq i\leq N,\quad W=\partial_w.
	\end{align}
	Given a domain $\Omega\subset \mathbbm{H}^N,$ for $u\in C^1(\Omega,\mathbb
	{R}),$ the subgradient (also known as the horizontal gradient) or the Heisenberg gradient $\nabla_{\mathbbm{H}^N} u$ is defined as
	\begin{align}
		\nabla_{\mathbbm{H}^N} u(\xi)\coloneqq\big(X_1u(\xi),\dots,X_{N}u(\xi),Y_1u(\xi),\dots,Y_Nu(\xi)\big).
	\end{align}
	Next, we explore the commutation relation between the vector fields. We have the following: 
	\begin{align}
		[X_i,Y_{i}]&=X_iY_{i}-Y_{i}X_i\\
		&=(\partial_{x_i}+2y_{i}\partial_w)(\partial_{y_{i}}-2x_i\partial_w)-(\partial_{y_{i}}-2x_i\partial_w)(\partial_{x_i}+2y_{i}\partial_w)\\
		&=-4\partial_w, \text{ for }1\leq i\leq N.
	\end{align}
	This implies that
	\begin{align}
		\text{rank}\big(\text{Lie}\{X_1,X_2,\dots,X_{2N},W\}(0,0)\big)=2N+1,
	\end{align}
	which is the Euclidean dimension of $\mathbbm{H}^N$. This is known as H\"ormander's rank condition in the literature. The Jacobian determinant of the dilation $\Phi_\lambda$ is $\lambda^Q$, where $Q=2N+2$ is the homogeneous dimension of $\mathbbm{H}^N.$
	Throughout this paper, we consider the following standard \textit{homogeneous norm} (see Definition \ref{homo}) on $\mathbbm{H}^N$$\colon$ 	\begin{align}\label{norm}
		\|\xi\|_{\mathbbm{H}^N}\text{:=}\,\bigg[\bigg(\sum_{i=1}^N \big(x_i^2+y_i^2\big)^2\bigg)+w^2\bigg]^{\frac{1}{4}},
	\end{align} 
	where 	
	\begin{align}
		\xi =(x_1,\dots,x_N,y_1,\dots,y_N,w)\in\mathbbm{H}^N.
	\end{align} This is referred to as Kor\'anyi norm in the Heisenberg group. The corresponding distance on $\mathbbm{H}^N$ is defined as
	\begin{align}\label{diss}
		d_{\mathbbm{H}^N}(\xi,\hat{\xi})\text{:=}\,\|\hat{\xi}^{-1}\circ\xi\|_{\mathbbm{H}^N},
	\end{align}
	where $\hat{\xi}^{-1}$ is the inverse of $\hat{\xi}$ with respect to the group operation $\circ,$ that is, $\hat{\xi}^{-1}=-\hat{\xi}.$

	Finally, the sub-Laplacian, also known as the Heisenberg Laplacian or Kohn-Laplace operator, is the self-adjoint operator defined as
	\begin{align}
		\Delta_{\mathbbm{H}^N}\coloneqq\sum_{i=1}^NX_i^2+Y_i^2.	
	\end{align}
	In terms of the Euclidean derivatives, it can be expressed as
	\begin{align}
		\Delta_{\mathbbm{H}^N}=\sum_{i=1}^N\frac{\partial^2}{\partial x_i^2}+\frac{\partial^2}{\partial y_i^2}+4y_i\frac{\partial^2}{\partial x_i\partial w}-4x_i\frac{\partial^2}{\partial y_i\partial w}+4\big(x_i^2+y_i^2\big)\frac{\partial^2}{\partial w^2}.
	\end{align} 
	Let us now define a $2N\times(2N+1)$ matrix whose rows are the coefficients of the vector fields $X_i$, i.e.,
	\begin{align}\label{sgm}
		\sigma=\begin{bmatrix}
			I_N &0 &2y\\
			0 &I_N &-2x
		\end{bmatrix}
	\end{align}
	for $x=(\xi_1,\dots,\xi_N)^\intercal$ and $y=(\xi_{N+1},\dots,\xi_{2N})^\intercal.$ The Heisenberg Laplacian can be also written as
	\begin{align}
		\Delta_{\mathbbm{H}^N}u=\text{div}\big(\sigma^\intercal\sigma\nabla u\big),
	\end{align}
	where $\nabla$ denotes the usual Euclidean gradient of $u.$ 
	
	The fractional sub-Laplacian operator is defined as follows$\colon$
	\begin{align}\label{frac}
		(-\Delta_{\mathbbm{H}^N})^su(\xi)=-\frac{1}{2}c(N,s)\int_{\mathbbm{H}^N}\frac{u(\xi\circ\eta)+u(\xi\circ\eta^{-1})-2u(\xi)}{\|\eta\|^{Q+2s}_{\mathbbm{H}^N}}d\eta,\,\,u\in S(\mathbbm{H}^N),\,\,\,\xi\in \mathbbm{H}^N, Q=2N+2,
	\end{align} 
	see  \cite[Section 5]{Palatucci} and  \cite[Proposition 1.4]{Ferrarii}. Here, $c(N,s)$ is a positive constant depending on the dimension $N$\,and the fractional parameter $s\in (0,1),$ and $S(\mathbbm{H}^N)$ denotes the Schwartz class of functions on $\mathbbm{H}^N$ (as detailed in  \cite[Chapter 1]{Folland}). It is worth mentioning that 
	$(-\Delta_{\mathbbm{H}^N})^s$ approximates $-\Delta_{\mathbbm{H}^N}$ as $s\To 1,$ see  \cite[Proposition 2.1.3]{Corni} and  \cite[Proposition 1.3]{Ferrarii} for the details.
	
	We recall from \cite{Ferrarii} that the heat semigroup associated with the fractional sub-Laplacian $(-\Delta_{\mathbbm{H}^N})^s$ is given by
	\begin{align}\label{hs}
		e^{-t(-\Delta_{\mathbbm{H}^N})^{s}}f(\xi)&\coloneqq f*h_s(\cdot,t)(\xi)\\
		&=\int_{\mathbbm{H}^N}h_s(t,\eta^{-1}\circ\xi)f(\eta)d\eta,
	\end{align}
	for $f\in L^2(\mathbbm{H}^N).$ Here, $h_s$ is the integral kernel of this semigroup. Moreover, for any $\xi\in\mathbbm{H}^N$ and $t>0,$ we have the following estimate on the heat kernel:
	\begin{align}\label{Esti}
		c^{-1}\min\bigg\{t^{-\frac{Q}{2s}},\frac{t}{\|\xi\|_{\mathbbm{H}^N}^{Q+2s}}\bigg\} \leq h_s(t,\xi)\leq c\min\bigg\{t^{-\frac{Q}{2s}},\frac{t}{\|\xi\|_{\mathbbm{H}^N}^{Q+2s}}\bigg\},
	\end{align}
	for some positive constant $c.$ Also, $$h_s(0,\cdot)=0 \text{ in }\mathbbm{H}^N.$$ The details can be found in  \cite[Section 2]{Ferrarii} and also in \cite{Chen P, Varo}. 
	Next, we define the notions of mild and weak solutions to \eqref{eq 0.1} as follows:
	\begin{defn}[Weak solution]\label{Weak}
		Let $u_0,f\in L^1_{\text{loc}}(\mathbbm{H}^N).$ Then $u\in L^p_{\text{loc}}((0,T),L^p_{\text{loc}}(\mathbbm{H}^N))$ is called a local-in-time weak solution to \eqref{eq 0.1} if 
		\begin{align}
			-\int_0^T\int_{\mathbbm{H}^N}u(\partial_t\phi-(-\Delta_{\mathbbm{H}^N})^s\phi) =\int_0^T\int_{\mathbbm{H}^N}|u|^p\phi +\int_{\mathbbm{H}^N}u_0\phi(0,\cdot)+\int_0^T\int_{\mathbbm{H}^N} f\phi,
		\end{align}
		holds for any non-negative test function $\phi\in C_c^1((0,T),C_c^2(\mathbbm{H}^N))$ with $\phi(\cdot,T)=0.$ 
		Moreover, $u$ is called a global-in-time weak solution if $T=+\infty.$ 
	\end{defn}
	
	\begin{defn}[Mild solution]\label{Mild}
		A function $u\in C([0,T],L^\infty(\mathbbm{H}^N))$ is called a local-in-time mild solution to \eqref{eq 0.1} if
		\begin{align}\label{eq Mild}
			u(t)=e^{-t(-\Delta_{\mathbbm{H}^N})^s}u_0+\int_0^te^{-(t-\tau)(-\Delta_{\mathbbm{H}^N})^s}(|u(\tau)|^p+f)d\tau
		\end{align}
		holds for any $t\in (0,T].$ Moreover, $u$ is called a global-in-time mild solution if $T=+\infty.$ 
	\end{defn}
	
	\begin{defn}\cite[Definition 2.1]{Palatucci}\label{homo}
		A homogeneous norm on $\mathbbm{H}^N$ is a continuous function (with respect to Euclidean topology)
		\begin{align}
			d_0\text{:}\,\mathbbm{H}^N\longrightarrow [0,\infty),
		\end{align}
		that satisfies the following properties$\colon$
		\begin{enumerate}
			\item [(i)] $d_0(\Phi_\lambda(\xi))=\lambda d_0(\xi),\text{ for all }\lambda>0\, \text{ and } \,\xi\in \mathbbm{H}^N,$
			\item [(ii)] $d_0(\xi)=0 \text{ if and only if }\xi=0.$
		\end{enumerate}
	\end{defn}
	\noindent Additionally, we say that the homogeneous norm is symmetric if
	\begin{align}
		d_0(\xi^{-1})=d_0(\xi),\text{ for all }\xi\in \mathbbm{H}^N.
	\end{align}
	In this paper, we use the homogeneous norm (Kor\'anyi norm) given by \eqref{norm} in $\mathbbm{H}^N.$ For fixed $\xi_0\in \mathbbm{H}^N\text{ and radius }\, R>0,$ the Kor\'anyi-ball of radius $R$ around $\xi_0$ is defined as
	\begin{align}\label{Ball}
		B_R^{\mathbbm{H}^N}(\xi_0)\text{:=}\,\bigg\{\xi\in \mathbbm{H}^N\,:\,\|\xi_0^{-1}\circ\xi\|_{\mathbbm{H}^N}<R\bigg\},
	\end{align}
	where $\|\cdot\|_{\mathbbm{H}^N}$ is the norm defined in \eqref{norm}. $B_R^{\mathbbm{H}^N}$ is also known as a \textit{Heisenberg-ball}.
	
	%
	%
	\section{Local well-posedness}
    \label{wellposedness}
	In this section, as mentioned earlier, our objective is to establish the local existence of solutions to \eqref{eq 0.1}. For this, we first prove the following auxiliary lemma.
	\begin{lem}
		The heat semigroup associated with $(-\Delta_{\mathbbm{H}^N})^{s}$ satisfies the following estimate:
		\begin{align}\label{infty}
			\|e^{-t(-\Delta_{\mathbbm{H}^N})^{s}}g\|_{L^\infty(\mathbbm{H}^N)}\leq \|g\|_{L^\infty(\mathbbm{H}^N)},
		\end{align}
		for any $g\in L^\infty(\mathbbm{H}^N)$ and $t>0.$
	\end{lem}
	\begin{proof}
		Recall that the heat semigroup associated with the fractional sub-Laplacian $(-\Delta_{\mathbbm{H}^N})^s$ is given by
		\begin{align}
			e^{-t(-\Delta_{\mathbbm{H}^N})^{s}}g(\xi)&\coloneqq g*h_s(t,\cdot)(\xi)\\
			&=\int_{\mathbbm{H}^N}h_s(t,\eta^{-1}\circ\xi)g(\eta)d\eta,
		\end{align}
		for $g\in L^\infty(\mathbbm{H}^N).$ In view of \cite[ expression (2.1)]{Ferrarii}, we have
		\begin{align}
			\|e^{-t(-\Delta_{\mathbbm{H}^N})^{s}}g\|_{L^\infty(\mathbbm{H}^N)}&=\bigg\|\int_{\mathbbm{H}^N} h_s(t,\eta^{-1}\circ\cdot)g(\eta)d\eta\bigg\|_{L^\infty(\mathbbm{H}^N)}\\&=\bigg\|\int_0^\infty f_{t,s}(\tau)e^{\tau\Delta_{\mathbbm{H}^N}}gd\tau \bigg\|_{L^\infty(\mathbbm{H}^N)},
		\end{align}
		where 
		\begin{align}
			f_{t,s}(\tau)\coloneqq\begin{cases}
				\frac{1}{2\pi\iota}\int_{\sigma-\iota\infty}^{\sigma+\iota\infty}e^{\tau z-tz^s}dz & \text{ for }\tau>0,\\[1mm]
				0 &\text{ for }\tau<0,
			\end{cases}
		\end{align}
		is introduced in  \cite[\S IX.11]{Yosida} for $\sigma>0,$ $t>0,$ and $s\in (0,1).$   Furthermore,
		\begin{align}
			\|e^{-t(-\Delta_{\mathbbm{H}^N})^{s}}g\|_{L^\infty(\mathbbm{H}^N)}		
			&\leq\sup_{t\geq 0}\|e^{t\Delta_{\mathbbm{H}^N}}g\|_{L^\infty(\mathbbm{H}^N)}\int_0^\infty f_{t,s}(\tau)d\tau\\
			&\leq \|g\|_{L^\infty(\mathbbm{H}^N)},
		\end{align} 
		follows by utilizing  \cite[Proposition 5]{Gre} and \cite[Proposition 11.3, Chapter IX]{Yosida}. Hence, we get the required estimate. 
	\end{proof}
	
	Now, we establish the key result of this section.
	
	\noindent \textbf{Proof of Theorem \ref{Local}.}
	(i) Firstly, for $$\delta(u_0,f)\coloneqq\max\{\|u_0\|_{L^\infty(\mathbbm{H}^N)}, \|f\|_{L^\infty(\mathbbm{H}^N)}\},$$ let us consider the following Banach space
	\begin{align}\label{e13}
		\Theta^s_T\coloneqq\bigg\{v\in C([0,T],L^\infty(\mathbbm{H}^N)):\|v\|_{L^\infty(\mathbbm{H}^N)}\leq 2\delta(u_0,f), v(0)=u_0\bigg\},
	\end{align}
	where $$\|v\|_{\Theta^s_T}=\|v\|_{L^\infty((0,T),L^\infty(\mathbbm{H}^N))}.$$
	For any $v\in \Theta_T^s,$ we define a map
	\begin{align}\label{Phii}
		\Phi(v)(t)\coloneqq e^{-t(-\Delta_{\mathbbm{H}^N})^s}v(0)+\int_0^te^{-(t-\tau)(-\Delta_{\mathbbm{H}^N})^s}(|v(\tau)|^p+f)d\tau.
	\end{align}
	We begin with proving that for any $v\in\Theta_T^s,$ we have that $\Phi(v)\in\Theta_T^s.$ From \eqref{infty} it implies 
	\begin{align}
		\|\Phi(v)\|_{\Theta_T^s}&\leq \|e^{-t(-\Delta_{\mathbbm{H}^N})^s}v(0)\|_{L^\infty(\mathbbm{H}^N)}+\Bigg\|\int_0^te^{-(t-\tau)(-\Delta_{\mathbbm{H}^N})^s}(|v(\tau)|^p+f)d\tau\Bigg\|_{L^\infty(H^N)}\\
		&\leq \|v(0)\|_{L^\infty(\mathbbm{H}^N)}+t(\|v\|^p_{L^\infty((0,t),L^\infty(\mathbbm{H}^N))}+\|f\|_{L^\infty(\mathbbm{H}^N)}). 
	\end{align}
	It is immediate to see that for $t=T>0,$ sufficiently small, one can get
	\begin{align}
		\|\Phi(v)\|_{\Theta_T^s}\leq 2\delta(u_0,f).
	\end{align}
	This yields $\Phi(v)\in \Theta_T^s.$
	We now proceed to show that for sufficiently small $T>0,$ $\Phi$ is a contraction map from $\Theta^s_T$ to itself. For given any functions $u_1,u_2\in \Theta^s_T,$ we have 
	\begin{align}
		\|\Phi(u_1)(t)-\Phi(u_2)(t)\|_{\Theta_T^s}&=\int_0^te^{-(t-\tau)(-\Delta_{\mathbbm{H}^N})^s}(|u(\tau)|^p-|v(\tau)|^p)d\tau\\
		&\leq 2^{p-1}C(p)\delta(u_0,f)^{p-1}T\|u_1-u_2\|_{L^\infty((0,T),L^\infty(\mathbbm{H}^N))}\\
		&\leq \|u_1-u_2\|_{L^\infty((0,T),L^\infty(\mathbbm{H}^N))},
	\end{align}
	for small enough $T>0.$
	Finally, an application of Banach fixed point theorem ensures the existence of a mild solution to \eqref{eq 0.1}. 
	
	Next, we show the uniqueness part. Let, if exist, $u_1,$ $u_2$ be two mild solutions to \eqref{eq 0.1} for some fixed $T>0$ with $u_1(0)=u_2(0)=u_0.$ Let us define
	\begin{align}\label{sup}
		\tau_0\coloneqq \sup\{t	
		\in [0,T] \text{ such that }u_1(\tau)=u_2(\tau)\, \forall \tau\in [0,t]\}.
	\end{align}
	Assume that $\tau_0\in [0,T).$ Notably, by the continuity of $u_1$ and $u_2$ in time, we have that
	\begin{align}
		u_1(\tau_0)=u_2(\tau_0).
	\end{align}
	Further, we define
	\begin{align}
		\widetilde{u}_1(t)\coloneqq u_1(t+\tau_0),\,\widetilde{u}_2(t)\coloneqq u_2(t+\tau_0).
	\end{align}
	Then, clearly $\widetilde{u}_i\in C([0,T-\tau_0],L^\infty(\mathbbm{H}^N))$ such that it satisfies \eqref{Phii} and $t\in (0,T-\tau_0]$ with $$\widetilde{u}_i(0)=u(\tau_0), \text{ for each } i=1,2.$$ 
	We proceed to show that for some positive constant $C=C(\tau')<1,$ for some $\tau'\in (0,T-\tau_0),$ we have
	\begin{align}\label{ln}
		\sup_{0<t<\tau'}\|\widetilde{u}_1(t)-\widetilde{u}_2(t)\|_{L^\infty(\mathbbm{H}^N)}\leq C\sup_{0<t<\tau'}\|\widetilde{u}_1(t)-\widetilde{u}_2(t)\|_{L^\infty(\mathbbm{H}^N)}.
	\end{align}
	This would immediately imply  that
	\begin{align}
		\widetilde{u}_1(t)=\widetilde{u}_2(t),\quad \forall \tau\in [0,\tau'].
	\end{align}
	This gives us a contradiction with \eqref{sup} that
	\begin{align}
		u_1(t+\tau_0)=u_2(t+\tau_0),\quad\forall t\in [0,\tau'].
	\end{align}
	We obtain using \eqref{infty} that
	\begin{align}
		\|\widetilde{u}_1(t)-\widetilde{u}_2(t)\|_{L^\infty(\mathbbm{H}^N)}&=\bigg\|\int_0^te^{-(t-t')(-\Delta_{\mathbbm{H}^N})^s}\big(|u_1(t')|^p-|u_2(t')|^p)\big)dt'\bigg\|\\
		&\leq 2^{{p-1}}tC(T,\tau_0,\widetilde{u}_1,\widetilde{u}_2)C(p)\delta(u_0,f)^{p-1}\sup_{0<t'<t}\|\widetilde{u}_1(t)-\widetilde{u}_2(t)\|_{L^\infty(\mathbbm{H}^N)},
	\end{align}
	and for $t$ small enough, we obtain the estimate \eqref{ln}.
	This concludes the proof.
	
	(ii) Utilizing the uniqueness of mild solutions to \eqref{eq 0.1} established above, we have the existence of the interval $[0,T_{\text{max}})$ using the well-known argument (see \cite{Caze}), where
	\begin{align}
		T_{\text{max}}\coloneqq \displaystyle{\sup_{T>0}}\{\eqref{eq 0.1} \text{ admits a solution }u\in C([0,T], L^\infty(\mathbbm{H}^N))\}.
	\end{align}
	Assume that $T_{\text{max}}$ is finite and that for each $t\in [0,T_{\text{max}}),$ we have
	\begin{align}
		\|u(t)\|_{L^\infty(\mathbbm{H}^N)}\leq C,
	\end{align} 
	for some constant $C>0.$ Fix $t^*\in (T_{\text{max}}/2,T_{\text{max}})$ and for $\widetilde{t} \in (0,T_{\text{max}}),$ let us consider a space
	\begin{align}
		\mathcal{M}\coloneqq\{v\in C([0,\widetilde{t}], L^\infty(\mathbbm{H}^N)): \|v\|_{C([0,\widetilde{t}), L^\infty(\mathbbm{H}^N))}<2\delta(C,\|f\|_{L^\infty(\mathbbm{H}^N)}),v(0)=u(t^*)\},
	\end{align}
	where $$\delta(C,\|f\|_{L^\infty(\mathbbm{H}^N)})=\max\{C,\|f\|_{L^\infty(\mathbbm{H}^N)}\}.$$ Similar to the arguments followed in the proof of (i) above, we define a map
	\begin{align}\label{Phii2}
		\Phi'(v)(t)\coloneqq e^{-t(-\Delta_{\mathbbm{H}^N})^s}u(t^*)+\int_0^te^{-(t-\tau)(-\Delta_{\mathbbm{H}^N})^s}(|v(\tau)|^p+f)d\tau,
	\end{align} 
	for $v\in \mathcal{M}$ and $t\in [0,\widetilde{t}].$ Like-wise (i), one can see that $\Phi':\mathcal{M}\To\mathcal{M}$ given by  \eqref{Phii2} is a contraction map. Therefore, the Banach fixed point theorem ensures the existence of a fixed point of $\mathcal{M},$ say $v.$ Furthermore, let us take $t^*$ such that
	\begin{align}
		\widetilde{t}+t^*>T_{\text{max}},
	\end{align}
	and consider
	\begin{align}\label{Ttt}
		\overline{u}(t)\coloneqq\begin{cases}
			u(t), &\text{ for }0\leq t\leq {t}^*\\
			v(t-{t}^*) &\text{ for }{t}^*\leq t\leq t^*+\widetilde{t}.
		\end{cases}
	\end{align}
	Observe that $\overline{u}\in C([0,t^*+\widetilde{t}], L^\infty(\mathbbm{H}^N)) $ is a solution of \eqref{eq 0.1}. This, together with \eqref{Ttt}, contradicts the definition of $T_{\text{max}}.$ Therefore, we get that
	\begin{align}
		\|u(t)\|_{L^\infty(\mathbbm{H}^N)}\To\infty
		\text{ as } t\To T_{\text{max}}.
	\end{align}	
	This concludes the proof.
	
	(iii) Instead of the space defined by \eqref{e13}, for
	\begin{align}
		\delta(u_0,f)\coloneqq\max\{\|u_0\|_{L^\infty(\mathbbm{H}^N)}, \|f\|_{L^\infty(\mathbbm{H}^N)}\},
	\end{align}
	and
	\begin{align}
		\delta_q(u_0,f)\coloneqq\max\{\|u_0\|_{L^q(\mathbbm{H}^N)}, \|f\|_{L^q(\mathbbm{H}^N)}\},
	\end{align} 
	consider a set
	\begin{align}
		\widetilde{\Theta}^s_T\coloneqq \bigg\{&v\in C([0,T_{\text{max}}),L^\infty(\mathbbm{H}^N))\cap C([0,T_{\text{max}}),L^q(\mathbbm{H}^N)):\\
		&\qquad\qquad\qquad\|v\|_{L^\infty((0,T_{\text{max}}),L^\infty(\mathbbm{H}^N))}\leq 2\delta(u_0,f),\, \|v\|_{L^\infty((0,T_{\text{max}}),L^q(\mathbbm{H}^N))}\leq 2\delta_q(u_0,f)\bigg\}.
	\end{align}
	We equip the space $\widetilde{\Theta}^s_T$ with the norm induced by the following distance:
	\begin{align}
		d_{q,\infty}(u,v)=\|u-v\|_{L^\infty((0,T_{\text{max}}),L^\infty(\mathbbm{H}^N))}+\|u-v\|_{L^\infty((0,T_{\text{max}}),L^q(\mathbbm{H}^N))},\,u,v\in \widetilde{\Theta}^s_T. 
	\end{align}
	In light of the inequality
	\begin{align}
		\||u(t)|^p\|_{L^q(\mathbbm{H}^N)}\leq \|u(t)\|^{p-1}_{L^\infty(\mathbbm{H}^N)}\|u(t)\|_{L^q(\mathbbm{H}^N)}, 
	\end{align}
	using the argument similar to (i) gives the existence of a unique solution $u\in \widetilde{\Theta}^s_T.$ Hence, $u\in C([0,T_{\text{max}}),L^\infty(\mathbbm{H}^N))\cap C([0,T_{\text{max}}),L^q(\mathbbm{H}^N)).$ Thus, (iii) follows. \qed 
	\\
	
	Next, we prove a one way relationship between the two notion of solutions- namely, the mild and weak solutions to \eqref{eq 0.1}. This demonstrates that the framework of weak solutions is more general than that of mild solutions. The result reads as follows:
	\begin{lem}\label{mild to weak}
		Let $u_0\in L^\infty(\mathbbm{H}^N).$ Let $u$ be a mild solution to \eqref{eq 0.1} in $\mathbbm{H}^N\times [0,T),$ for $T>0.$ Then $u$ is also a weak solution to \eqref{eq 0.1} in $\mathbbm{H}^N\times [0,T).$
	\end{lem}
	\begin{proof}
		Let $f\in L^1(\mathbbm{H}^N).$ Since $u$ is a mild solution, by this, we have using \eqref{eq Mild} that
		\begin{align}\label{l1}
			u(t)= e^{-t(-\Delta_{\mathbbm{H}^N})^s}u_0+\int_0^te^{-(t-\tau)(-\Delta_{\mathbbm{H}^N})^s}(|u(\tau)|^p+f)d\tau.
		\end{align}
		In order to prove the claim, let us consider a test function $\phi\in C_0^1((0,T),C_0^2(\mathbbm{H}^N))$ with $\phi(\cdot, T)=0.$ We multiply both sides of the relation \eqref{l1} with 
		$\phi$ followed by integrating over $\mathbbm{H}^N$ to get
		\begin{align}\label{n,}
			\int_{\mathbbm{H}^N}u\phi=\int_{\mathbbm{H}^N}e^{-t(-\Delta_{\mathbbm{H}^N})^s}u_0\phi+\int_\mathbbm{H^N}\bigg(\int_0^te^{-(t-\tau)(-\Delta_{\mathbbm{H}^N})^s}(|u(\tau)|^p+f)d\tau\bigg)\phi.
		\end{align}
		Here, we recall that $(-\Delta_{\mathbbm{H}^N})^s$ is a self-adjoint operator, that is,
		\begin{align}\label{n,1}
			\int_{\mathbbm{H}^N}v(-\Delta_{\mathbbm{H}^N})^su=\int_{\mathbbm{H}^N}u(-\Delta_{\mathbbm{H}^N})^sv,
		\end{align}
		for all $u,v\in H^s(\mathbbm{H}^N),$ see \cite{Corni}.
		We take the time derivative of each term in \eqref{n,} as follows:
		\begin{align}
			\frac{\partial}{\partial t}\int_{\mathbbm{H}^N}u\phi&=\frac{\partial}{\partial t}\int_{\mathbbm{H}^N}e^{-t(-\Delta_{\mathbbm{H}^N})^s}u_0\phi+\frac{\partial}{\partial t}\int_{\mathbbm{H}^N}\bigg(\int_0^te^{-(t-\tau)(-\Delta_{\mathbbm{H}^N})^s}(|u(\tau)|^p+f)d\tau\bigg)\phi\\
			&=\int_{\mathbbm{H}^N}\frac{\partial}{\partial t}\bigg(e^{-t(-\Delta_{\mathbbm{H}^N})^s}u_0\phi\bigg)+\int_{\mathbbm{H}^N}\frac{\partial}{\partial t}\bigg(\bigg(\int_0^te^{-(t-\tau)(-\Delta_{\mathbbm{H}^N})^s}(|u(\tau)|^p+f)d\tau\bigg)\phi\bigg)\\
			&=-\int_{\mathbbm{H}^N}(-\Delta_{\mathbbm{H}^N})^se^{-t(-\Delta_{\mathbbm{H}^N})^s}u_0\phi+\int_{\mathbbm{H}^N}e^{-t(-\Delta_{\mathbbm{H}^N})^s}u_0\partial_t\phi\\
			&\qquad-\int_{\mathbbm{H}^N}\left((-\Delta_{\mathbbm{H}^N})^s\int_0^te^{-(t-\tau)(-\Delta_{\mathbbm{H}^N})^s}(|u(\tau)|^p+f)d\tau\right)\phi\\
			&\qquad+\int_{\mathbbm{H}^N}\left(\int_0^te^{-(t-\tau)(-\Delta_{\mathbbm{H}^N})^s}(|u|^p+f)\right)\partial_t\phi+\int_{\mathbbm{H}^N}(|u(t)|^p+f)\phi(t).
		\end{align}
		Further, utilizing \eqref{n,1} accords with
		\begin{align}
			\frac{\partial}{\partial t}\int_{\mathbbm{H}^N}u\phi&=-\int_{\mathbbm{H}^N}e^{-t(-\Delta_{\mathbbm{H}^N})^s}u_0\left((-\Delta_{\mathbbm{H}^N})^s\phi\right)+\int_{\mathbbm{H}^N}e^{-t(-\Delta_{\mathbbm{H}^N})^s}u_0\partial_t\phi\\
			&\qquad-\int_{\mathbbm{H}^N}\left(\int_0^te^{-(t-\tau)(-\Delta_{\mathbbm{H}^N})^s}(|u(\tau)|^p+f)d\tau\right)(-\Delta_{\mathbbm{H}^N})^s\phi\\
			&\qquad+\int_{\mathbbm{H}^N}\left(\int_0^te^{-(t-\tau)(-\Delta_{\mathbbm{H}^N})^s}(|u(\tau)|^p+f)d\tau\right)\partial_t\phi+\int_{\mathbbm{H}^N}(|u(t)|^p+f)\phi(t).
		\end{align}
		Moreover, making use of \eqref{l1} in the above equation deduces
		\begin{align}
			\frac{\partial}{\partial t}\int_{\mathbbm{H}^N}u\phi&=\int_{\mathbbm{H}^N}(-\Delta_{\mathbbm{H}^N})^s\phi\bigg(-e^{-t(-\Delta_{\mathbbm{H}^N})^s}u_0
			-\int_0^te^{-(t-\tau)(-\Delta_{\mathbbm{H}^N})^s}(|u(\tau)|^p+f)d\tau\bigg)\\
			&\qquad+\int_{\mathbbm{H}^N}\partial_t\phi\bigg(e^{-t(-\Delta_{\mathbbm{H}^N})^s}u_0
			+\int_0^te^{-(t-\tau)(-\Delta_{\mathbbm{H}^N})^s}(|u(\tau)|^p+f)d\tau\bigg)+\int_{\mathbbm{H}^N}(|u(t)|^p+f)\phi(t)\\
			&=-\int_{\mathbbm{H}^N}u(-\Delta_{\mathbbm{H}^N})^s\phi+\int_{\mathbbm{H}^N}u\partial_t\phi+\int_{\mathbbm{H}^N}(|u(t)|^p+f)\phi(t).
		\end{align} 
		Finally, we integrate the above equation from $0$ to $T,$ and obtain
		\begin{align}
			-\int_{\mathbbm{H}^N}u_0\phi(0,\xi)=-\int_0^T\int_{\mathbbm{H}^N}u(-\Delta_{\mathbbm{H}^N})^s\phi+\int_0^T\int_{\mathbbm{H}^N}u\partial_t\phi+\int_0^T\int_{\mathbbm{H}^N}(|u|^p+f)\phi,
		\end{align}
		equivalently,
		\begin{align}
			-\int_0^T\int_{\mathbbm{H}^N}u\partial_t\phi+\int_0^T\int_{\mathbbm{H}^N}u(-\Delta_{\mathbbm{H}^N})^s\phi=\int_0^T\int_{\mathbbm{H}^N}(|u|^p+f)\phi+\int_{\mathbbm{H}^N}u_0\phi(0,\cdot).
		\end{align}
		Hence, in light of \eqref{Weak}, the above expression infers that $u$ is a weak solution to \eqref{eq 0.1}. This proves the required result.
	\end{proof}
	
	\section{Fujita exponent}
    \label{Fujitaexponent}
	We now proceed to prove the results related to existence of global-in-time solutions. To establish the main results of this  section, we will rely on the following lemmas. The proofs follow the same conceptual framework outlined in  \cite[Proposition 2.1--2.2]{ZhangQ}.
	\begin{lem}\label{Lem Q1}
		For any $\delta>0,$ there exists a constant $C_1=C_1(Q,\delta)>0$ such that
		\begin{align}
			\sup_{\xi\in\mathbbm{H}^N}\int_{\mathbbm{H}^N}\frac{1}{{d_{\mathbbm{H}^N}(\eta,\xi)^{Q-2s}(1+d_{\mathbbm{H}^N}(\eta)^{Q+\delta})}}d\eta\leq C_1.		
		\end{align}
	\end{lem}
	\begin{proof}
		We recall that the Haar measure on $\mathbbm{H}^N$ is translation invariant, see, for instance  \cite[Section 2.3]{Gre}. Making use of the polar coordinates (see, e.g. \cite[Proposition 1.15]{Folland}), we observe that
		\begin{align}
			\int_{\mathbbm{H}^N}\frac{1}{{d_{\mathbbm{H}^N}(\eta)^{Q-2s}(1+d_{\mathbbm{H}^N}(\eta)^{Q+\delta})}}d\eta&=\int_0^\infty \frac{r^{Q-1}}{r^{Q-2s}(1+r^{Q+\delta})}dr\\
			&= \int_0^\infty \frac{r^{-1+2s}}{(1+r^{Q+\delta})}dr\\
			&\leq C_1,
		\end{align}
		since $0<s<1$ and $Q\geq 3.$
	\end{proof}
	
	Using similar arguments as in Lemma \ref{Lem Q1}, we have the following results:
	\begin{lem}\label{Lem Q5}
		For any $\delta>0,$ there exists a constant $C_2=C_2(Q,\delta)>0$ such that
		\begin{align}
			\sup_{\xi\in\mathbbm{H}^N}\int_{\mathbbm{H}^N}\frac{1}{{d_{\mathbbm{H}^N}(\eta,\xi)^{Q-2s}(1+d_{\mathbbm{H}^N}(\eta)^{2s+\delta})}}d\eta\leq C_2.		
		\end{align}
	\end{lem}
	
	\begin{lem}\label{Lem Q3}
		For any $\delta>0,$ there exists a constant $C_3=C_3(Q,\delta)>0$ such that
		\begin{align}
			\sup_{\xi\in\mathbbm{H}^N}\int_{\mathbbm{H}^N}\frac{d_{\mathbbm{H}^N}(\eta)^{Q-1}}{(1+d_{\mathbbm{H}^N}(\eta)^{Q+s+\delta})d_{\mathbbm{H}^N}(\eta,\xi)^{Q+2s}}d\eta\leq C_3.		
		\end{align}
	\end{lem}
	
	\begin{lem}\label{Lem Q4}
		For any $\delta>0,$ there exists a constant $C_4=C_4(Q,\delta)>0$ such that
		\begin{align}
			\sup_{\xi\in\mathbbm{H}^N}\int_{\mathbbm{H}^N}\frac{d_{\mathbbm{H}^N}(\eta)^{Q-1}}{(1+d_{\mathbbm{H}^N}(\eta)^{2Q-s+\delta})d_{\mathbbm{H}^N}(\eta,\xi)^{Q+2s}}d\eta\leq C_4.		
		\end{align}
	\end{lem}

	\begin{lem}\label{Lem Q2}
		For any $\delta>0,$ there exists a constant $C_5>0$ depending only on $Q$ and $\delta$ such that
		\begin{align}
			\int_{\mathbbm{H}^N}\frac{1}{{d_{\mathbbm{H}^N}(\eta,\xi)^{Q-2s}(1+d_{\mathbbm{H}^N}(\eta)^{Q+\delta})}}d\eta\leq \frac{C_5}{1+d_{\mathbbm{H}^N}(\xi)^{Q-2s}}\,\, \forall \xi\in\mathbbm{H}^N.
		\end{align}
	\end{lem}
	\begin{proof}
		We split the proof into two parts. First, let $d_{\mathbbm{H}^N}(\xi)\geq 1.$ In this case, using the pseudo triangle inequality $d_{\mathbbm{H}^N}(\xi)\leq c(d_{\mathbbm{H}^N}(\eta,\xi)+d_{\mathbbm{H}^N}(\eta)),$ we deduce
		\begin{align}
			&\int_{\mathbbm{H}^N}\frac{1}{{d_{\mathbbm{H}^N}(\eta,\xi)^{Q-2s}(1+d_{\mathbbm{H}^N}(\eta)^{Q+\delta})}}d\eta\\
			&\qquad=\int_{\left\{d_{\mathbbm{H}^N}(\eta)\geq \frac{d_{\mathbbm{H}^N}(\xi)}{2c}\right\}}\frac{1}{{d_{\mathbbm{H}^N}(\eta,\xi)^{Q-2s}(1+d_{\mathbbm{H}^N}(\eta)^{Q+\delta})}}d\eta\\
			&\qquad\qquad\qquad\quad+\int_{\left\{d_{\mathbbm{H}^N}(\eta)\leq \frac{d_{\mathbbm{H}^N}(\xi)}{2c}\right\}}\frac{1}{{d_{\mathbbm{H}^N}(\eta,\xi)^{Q-2s}(1+d_{\mathbbm{H}^N}(\eta)^{Q+\delta})}}d\eta\\
			&\qquad\leq C\int_{\left\{d_{\mathbbm{H}^N}(\eta)\geq \frac{d_{\mathbbm{H}^N}(\xi)}{2c}\right\}}\frac{1}{{d_{\mathbbm{H}^N}(\eta,\xi)^{Q-2s}(1+d_{\mathbbm{H}^N}(\eta)^{2s+\delta})(1+d_{\mathbbm{H}^N}(\xi)^{Q-2s})}}d\eta\\
			&\quad\qquad\qquad\qquad+\int_{\left\{d_{\mathbbm{H}^N}(\eta)\leq \frac{d_{\mathbbm{H}^N}(\xi)}{2c}\right\}}\frac{1}{{d_{\mathbbm{H}^N}(\xi,\eta)^{Q-2s}(1+d_{\mathbbm{H}^N}(\eta)^{Q+\delta})}}d\eta.		
		\end{align}
		Further, utilizing Lemma \ref{Lem Q5} gives
		\begin{align}\label{r1}
			\int_{\mathbbm{H}^N}\frac{1}{{d_{\mathbbm{H}^N}(\eta,\xi)^{Q-2s}(1+d_{\mathbbm{H}^N}(\eta)^{Q+\delta})}}d\eta&\leq \frac{C_2}{(1+d_{\mathbbm{H}^N}(\xi)^{Q-2s})}+\frac{C}{d_{\mathbbm{H}^N}(\xi)^{Q-2s}}\\
			&\leq \frac{2C}{(1+d_{\mathbbm{H}^N}(\xi)^{Q-2s})},
		\end{align}
		for some constant $C>0.$ For $d_{\mathbbm{H}^N}(\xi)\leq 1,$ we have by Lemma \ref{Lem Q1} that
		\begin{align}\label{r2}
			\int_{\mathbbm{H}^N}\frac{1}{{d_{\mathbbm{H}^N}(\eta,\xi)^{Q-2s}(1+d_{\mathbbm{H}^N}(\eta)^{Q+\delta})}}d\eta\leq C_1,
		\end{align}
		for some constant $C_1>0,$ whence altogether we obtain the desired result. 
	\end{proof}

	Also, it is immediate to observe the following:
\begin{lem}\label{Lem Q}
	For any $\delta>0,$ there exists a constant $C_6=C_6(Q,\delta)$ such that
\begin{align}
		\int_{\mathbbm{H}^N}\frac{d_{\mathbbm{H}^N}(\eta)^{Q-1}}{1+d_{\mathbbm{H}^N}(\eta)^{2Q-s+\delta}}d\eta\leq C_6.
\end{align}
\end{lem}
    
	We are now ready to give the proof of Theorem
	\ref{Thm Glo2}.
	
	
	\noindent \textbf{Proof of Theorem \ref{Thm Glo2}.}
	We use a fixed point argument on the following integral equation:
	\begin{align}\label{y1}
		u(t)=e^{-t(-\Delta_{\mathbbm{H}^N})^s}u_0+\int_0^te^{-(t-\tau)(-\Delta_{\mathbbm{H}^N})^s}(|u(\tau)|^p+f)d\tau		
	\end{align}
	with $\|u_0\|_{L^\infty(\mathbbm{H}^N)}\leq \varepsilon$ for small enough $\varepsilon>0$ that we fix later. It is an easy observation that
	\begin{align}\label{d1}
		||u_1|^p-|u_2|^p|\leq C|u_1-v_1|\left(|u_1|^{p-1}+|v_1|^{p-1}\right),
	\end{align}
	for some positive constant $C$ and $p> \frac{Q}{Q-2s}.$
	In order to employ the fixed point arguments, we first define a suitable complete metric space given as follows:
	\begin{align}\label{psc}
		\Theta^s_M\coloneqq\bigg\{v\in C([0,\infty),C_c(\mathbbm{H}^N)):0\leq v(t,\xi)\leq \frac{M}{1+d_{\mathbbm{H}^N}(\xi)^{Q-2s}} \text{ for }t\in [0,\infty),\, \xi\in\mathbbm{H}^N \bigg\},
	\end{align}
	where $$\|v\|_{\Theta^s_T}=\|v\|_{L^\infty((0,T),C_c(\mathbbm{H}^N))}$$ and $M>0$ is a constant. 
	We recall \cite{Talwar}, where the authors defined a similar space to study the Fujita theory for sub-elliptic heat equation in a stratified Lie group setting. Now, for any $u\in \Theta^s_M,$ we define the following map:
	\begin{align}\label{Phii1}
		\Phi_M(u)(t)\coloneqq e^{-t(-\Delta_{\mathbbm{H}^N})^s}u_0+\int_0^te^{-(t-\tau)(-\Delta_{\mathbbm{H}^N})^s}\big(|u(\tau)|^p+f\big)d\tau.
	\end{align}
	We aim to show that $\Phi_M$ is a contraction mapping for some $M>0.$ For this, in view of the fact that $p>\frac{Q}{Q-2s},$ we fix 
	\begin{align}\label{choice f}
		u_0(\xi)=\frac{\varepsilon d_{\mathbbm{H}^N}(\xi)^{Q-1}}{1+d_{\mathbbm{H}^N}(\xi)^{2Q-s+\delta}}\, \text{ and } f(\xi)=\frac{\varepsilon}{1+d_{\mathbbm{H}^N}(\xi)^{Q+\delta}},
	\end{align}
	where $\delta>0$ is chosen such that
	\begin{align}\label{delt}
		\frac{1}{(1+d_{\mathbbm{H}^N}(\xi))^{(Q-2s)p}}&\leq \frac{1}{(1+d_{\mathbbm{H}^N}(\xi))^{Q+\delta}}\\
		&\leq \frac{C}{(1+d_{\mathbbm{H}^N}(\xi)^{Q+\delta})},
	\end{align}
	for some positive constant $C$ independent of $\xi.$
	Firstly, using \eqref{Esti} and in view of Lemma \ref{Lem Q}, we have that
	\begin{align}\label{cd}
		\qquad e^{-t(-\Delta_{\mathbbm{H}^N})^s}u_0(\xi)&=\int_{\mathbbm{H}^N}h_s(t,\eta^{-1}\circ\xi)u_0(\eta)d\eta\\
		&=\int_{\mathbbm{H}^N}\frac{\varepsilon  h_s(t,\eta^{-1}\circ\xi)d_{\mathbbm{H}^N}(\eta)^{Q-1}}{1+d_{\mathbbm{H}^N}(\eta)^{2Q-s+\delta}}d\eta\\
		&\leq{\varepsilon C}\int_{\mathbbm{H}^N}\frac{d_{\mathbbm{H}^N}(\eta)^{Q-1}}{(1+d_{\mathbbm{H}^N}(\eta)^{2Q-s+\delta})d_{\mathbbm{H}^N}(\eta,\xi)^{Q+2s}}d\eta\\
		&={\varepsilon C}\int_{\left\{d_{\mathbbm{H}^N}(\eta)>\frac{d_{\mathbbm{H}^N}(\xi)}{2c}\right\}}\frac{d_{\mathbbm{H}^N}(\eta)^{Q-1}}{(1+d_{\mathbbm{H}^N}(\eta)^{2Q-s+\delta})d_{\mathbbm{H}^N}(\eta,\xi)^{Q+2s}}d\eta\\
		&\qquad+{\varepsilon C}\int_{\left\{d_{\mathbbm{H}^N}(\eta)\leq\frac{d_{\mathbbm{H}^N}(\xi)}{2c}\right\}}\frac{d_{\mathbbm{H}^N}(\eta)^{Q-1}}{(1+d_{\mathbbm{H}^N}(\eta)^{2Q-s+\delta})d_{\mathbbm{H}^N}(\eta,\xi)^{Q+2s}}d\eta\\
		&\leq {\varepsilon C}\int_{\left\{d_{\mathbbm{H}^N}(\eta)>\frac{d_{\mathbbm{H}^N}(\xi)}{2c}\right\}}\frac{d_{\mathbbm{H}^N}(\eta)^{Q-1}}{(1+d_{\mathbbm{H}^N}(\eta)^{Q-2s})(1+d_{\mathbbm{H}^N}(\eta)^{Q+s+\delta})d_{\mathbbm{H}^N}(\eta,\xi)^{Q+2s}}d\eta\\
		&\qquad+{\varepsilon C}\int_{\left\{d_{\mathbbm{H}^N}(\eta)\leq\frac{d_{\mathbbm{H}^N}(\xi)}{2c}\right\}}\frac{d_{\mathbbm{H}^N}(\eta)^{Q-1}}{(1+d_{\mathbbm{H}^N}(\eta)^{2Q-s+\delta})d_{\mathbbm{H}^N}(\eta,\xi)^{Q+2s}}d\eta\\ 
		&\leq {\varepsilon C}\frac{1}{1+{\left(\frac{d_{\mathbbm{H}^N}(\xi)}{2c}\right)}^{Q-2s}}\int_{\left\{d_{\mathbbm{H}^N}(\eta)>\frac{d_{\mathbbm{H}^N}(\xi)}{2c}\right\}}\frac{d_{\mathbbm{H}^N}(\eta)^{Q-1}}{(1+d_{\mathbbm{H}^N}(\eta)^{Q+s+\delta})d_{\mathbbm{H}^N}(\eta,\xi)^{Q+2s}}d\eta\\
		&\qquad+{\varepsilon C}\int_{\left\{d_{\mathbbm{H}^N}(\eta)\leq\frac{d_{\mathbbm{H}^N}(\xi)}{2c}\right\}}\frac{d_{\mathbbm{H}^N}(\eta)^{Q-1}}{(1+d_{\mathbbm{H}^N}(\eta)^{2Q-s+\delta})d_{\mathbbm{H}^N}(\eta,\xi)^{Q+2s}}d\eta.
	\end{align}
	Using Lemma \ref{Lem Q3} in \eqref{cd}, we deduce
	\begin{align}
		e^{-t(-\Delta_{\mathbbm{H}^N})^s}u_0(\xi)&\leq \frac{\varepsilon CC_3}{\left(1+{\left(\frac{d_{\mathbbm{H}^N}(\xi)}{2c}\right)}^{Q-2s}\right)}\\
		&\qquad\qquad+{\varepsilon C}\int_{\left\{d_{\mathbbm{H}^N}(\eta)\leq\frac{d(\xi)}{2c}\right\}}\frac{d_{\mathbbm{H}^N}(\eta)^{Q-1}}{(1+d_{\mathbbm{H}^N}(\eta)^{2Q-s+\delta})d_{\mathbbm{H}^N}(\eta,\xi)^{Q+2s}}d\eta
	\end{align}
	Now, to estimate $e^{-t(-\Delta_{\mathbbm{H}^N})^s}u_0(\xi),$ we proceed to do it in two cases. First, we consider that $d_{\mathbbm{H}^N}(\xi)>1,$ and observe
	\begin{align}
		e^{-t(-\Delta_{\mathbbm{H}^N})^s}u_0(\xi)&\leq {\varepsilon C}\frac{1}{1+\left({\frac{d_{\mathbbm{H}^N}(\xi)}{2c}}\right)^{Q-2s}}\\
		&\qquad\qquad\quad+{\varepsilon C}\int_{\left\{d_{\mathbbm{H}^N}(\eta)\leq\frac{d_{\mathbbm{H}^N}(\xi)}{2c}\right\}}\frac{d_{\mathbbm{H}^N}(\eta)^{Q-1}}{(1+d_{\mathbbm{H}^N}(\eta)^{2Q-s+\delta})d_{\mathbbm{H}^N}(\xi)^{Q+2s}}d\eta\\
		&\leq \frac{\varepsilon C}{1+{{d_{\mathbbm{H}^N}(\xi)}}^{Q-2s}}\\
		&\qquad\qquad\quad+\frac{\varepsilon C}{d_{\mathbbm{H}^N}(\xi)^{Q+2s}}\int_{\left\{d_{\mathbbm{H}^N}(\eta)\leq\frac{d_{\mathbbm{H}^N}(\xi)}{2c}\right\}}\frac{d_{\mathbbm{H}^N}(\eta)^{Q-1}}{(1+d_{\mathbbm{H}^N}(\eta)^{2Q-s+\delta})}d\eta,
	\end{align}
	for some constant $C$ independent of $\xi.$	
	Then, for some constant $C>0$ it infers
	\begin{align}
		e^{-t(-\Delta_{\mathbbm{H}^N})^s}u_0(\xi)&\leq\frac{\varepsilon C}{1+d_{\mathbbm{H}^N}(\xi)^{Q-2s}}+\frac{\varepsilon C}{d_{\mathbbm{H}^N}(\xi)^{Q+2s}}.
	\end{align}
	Now, again using the condition $d_{\mathbbm{H}^N}(\xi)>1$ above accords with the following estimate:
	\begin{align}\label{eaw1}
		e^{-t(-\Delta_{\mathbbm{H}^N})^s}u_0(\xi)&\leq\frac{\varepsilon C}{1+d_{\mathbbm{H}^N}(\xi)^{Q-2s}}+\frac{\varepsilon C}{d_{\mathbbm{H}^N}(\xi)^{Q+2s}}\\
		&\leq\frac{\varepsilon C}{1+d_{\mathbbm{H}^N}(\xi)^{Q-2s}}+\frac{\varepsilon C}{d_{\mathbbm{H}^N}(\xi)^{Q-2s}}\\
		&\leq \frac{2 C}{1+d_{\mathbbm{H}^N}(\xi)^{Q-2s}}.
	\end{align}
	Next, if $d_{\mathbbm{H}^N}(\xi)\leq 1,$ we consider a set
	\begin{align}
		S=\{\xi\in\mathbbm{H}^N:d_{\mathbbm{H}^N}(\xi)\leq 1\}.	
	\end{align}
	Clearly, $S$ is compact and hence the term
	\begin{align}
		\int_{\left\{d_{\mathbbm{H}^N}(\eta)\leq\frac{d_{\mathbbm{H}^N}(\xi)}{2c}\right\}}\frac{d_{\mathbbm{H}^N}(\eta)^{Q-1}}{(1+d_{\mathbbm{H}^N}(\eta)^{2Q-s+\delta})d_{\mathbbm{H}^N}(\eta,\xi)^{Q+2s}}d\eta
	\end{align}
	is bounded by some constant $C$ being continuous function of $\xi.$ This infers using Lemmas \ref{Lem Q3} and \ref{Lem Q4} that 
	\begin{align}\label{eaw2}
		e^{-t(-\Delta_{\mathbbm{H}^N})^s}u_0(\xi)&\leq  {\varepsilon C}\frac{1}{1+{\left(\frac{d_{\mathbbm{H}^N}(\xi)}{2c}\right)}^{Q-2s}}\int_{\left\{d_{\mathbbm{H}^N}(\eta)>\frac{d_{\mathbbm{H}^N}(\xi)}{2c}\right\}}\frac{d_{\mathbbm{H}^N}(\eta)^{Q-1}}{(1+d_{\mathbbm{H}^N}(\eta)^{Q+s+\delta})d_{\mathbbm{H}^N}(\eta,\xi)^{Q+2s}}d\eta\\
		&\qquad+{\varepsilon C}\int_{\left\{d_{\mathbbm{H}^N}(\eta)\leq\frac{d_{\mathbbm{H}^N}(\xi)}{2c}\right\}}\frac{d_{\mathbbm{H}^N}(\eta)^{Q-1}}{(1+d_{\mathbbm{H}^N}(\eta)^{2Q-s+\delta})d_{\mathbbm{H}^N}(\eta,\xi)^{Q+2s}}d\eta\\
		&\leq {\varepsilon C}\frac{C_3}{1+{\left(\frac{d_{\mathbbm{H}^N}(\xi)}{2c}\right)}^{Q-2s}}+{\varepsilon CC_4}\\
		&\leq {\varepsilon C}\frac{C_3}{1+{{d_{\mathbbm{H}^N}(\xi)}}^{Q-2s}}+\varepsilon CC_4\\
		&\leq C,
	\end{align}
	for some positive constant $C,$ as $d_{\mathbbm{H}^N}(\xi)\leq 1.$ Therefore, in view on the estimates \eqref{eaw1} and \eqref{eaw2}, we have that there exists a constant $C>0$ such that
	\begin{align}
		e^{-t(-\Delta_{\mathbbm{H}^N})^s}u_0(\xi)\leq \frac{2 C}{1+d_{\mathbbm{H}^N}(\xi)^{Q-2s}},\forall \xi\in\mathbbm{H}^N \text{ and }t\in [0,\infty).
	\end{align}
	Similarly, we next estimate the second term in \eqref{y1}, i.e., $\int_0^te^{-(t-\tau)(-\Delta_{\mathbbm{H}^N})^s}(|u(\tau)|^p+f)d\tau.$ 
	In light of \eqref{psc} and \eqref{choice f}, we have
	\begin{align}\label{ulin}
		\int_0^te^{-(t-\tau)(-\Delta_{\mathbbm{H}^N})^s}&(|u(\tau)|^p+f)(\xi)d\tau\\
		&\qquad=\int_0^t\int_{\mathbbm{H}^N}h_{s}(t-\tau,\eta^{-1}\circ\xi)(|u(\tau,\eta)|^p+f(\eta))d\eta d\tau\\
		&\qquad\leq  \int_0^t\int_{\mathbbm{H}^N}h_{s}(t-\tau,\eta^{-1}\circ\xi)\left(\frac{M^p}{(1+d_{\mathbbm{H}^N}(\eta)^{Q-2s})^p}+\frac{\varepsilon}{1+d_{\mathbbm{H}^N}(\eta)^{Q+\delta}}\right)d\eta d\tau.
	\end{align}
	Since $p>\frac{Q}{Q-2s},$ this further yields
	\begin{align}
		\int_0^te^{-(t-\tau)(-\Delta_{\mathbbm{H}^N})^s}(|u(\tau)|^p&+f)d\tau\\
		&\leq \int_0^t\int_{\mathbbm{H}^N}h_{s}(t-\tau,\eta^{-1}\circ\xi)\left(\frac{CM^p}{1+d_{\mathbbm{H}^N}(\eta)^{Q+\delta}}+\frac{\varepsilon}{1+d_{\mathbbm{H}^N}(\eta)^{Q+\delta}}\right)d\eta d\tau,	
	\end{align}
	where $\delta$ is the same constant appearing in \eqref{delt}. By the definition of the heat kernel $h_s(t,\xi),$ we have that
	\begin{align}
		\frac{\partial h_{s}(t, \eta^{-1}\circ\xi)}{\partial t}+(-\Delta_{\mathbbm{H}^N})^sh_{s}(t, \eta^{-1}\circ\xi)=0,
	\end{align}
	for all $\xi,\,\eta\in\mathbbm{H}^N.$ This further accords on integrating over the time interval $[0,T]$ that
	\begin{align}
		\int_0^T\bigg(\frac{\partial h_{s}(t, \eta^{-1}\circ\xi)}{\partial t}+(-\Delta_{\mathbbm{H}^N})^sh_{s}(t, \eta^{-1}\circ\xi)\bigg)dt=0\,\, \forall \xi,\,\eta\in\mathbbm{H}^N.
	\end{align}
	Now, taking the limit $T\To\infty$ and using the estimate \eqref{Esti} gives
	\begin{align}
		0&=\lim_{T\To\infty}\int_0^T (-\Delta_{\mathbbm{H}^N})^sh_{s}(t, \eta^{-1}\circ\xi)dt\\
		&=\lim_{T\To\infty}(-\Delta_{\mathbbm{H}^N})^s\int_0^T h_{s}(t, \eta^{-1}\circ\xi)dt.
	\end{align}
	Next, the non-negativity of $h_s$ infers
	\begin{align}\label{ul}
		\int_0^Th_{s}(T-t,\eta^{-1}\circ\xi) dt
		&\leq \lim_{T\To\infty}\int_0^Th_{s}(T-t,\eta^{-1}\circ\xi) dt\\ &=P_{2s}(\eta^{-1}\circ\xi)\\
		&\leq \frac{C}{\|\eta^{-1}\circ\xi\|_{\mathbbm{H}^N}^{Q-2s}}\\
		&=\frac{C}{d_{\mathbbm{H}^N}(\eta,\xi)^{Q-2s}},
	\end{align}
	for some positive constant independent of $\xi$ and $\eta.$ This follows in light of  \cite[Proposition 4.1]{Garofa} (also \cite[Remark 2.14 (ii)]{Corni}), where
	\begin{align}
		P_{2s}(\xi)\coloneqq \frac{1}{\Gamma (s)}\int_0^\infty t^{s-1}h(t,\xi)dt
	\end{align}
	is the fundamental solution of $-(-\Delta_{\mathbbm{H}^N})^s$ in $\mathbbm{H}^N\setminus\{0\}.$ Hence, we have that
	\begin{align}
		(-\Delta_{\mathbbm{H}^N})^sP_{2s}(\xi)=0 \text{ for any }\xi\in\mathbbm{H}^N\setminus\{0\}.
	\end{align}
	The details can be found in \cite[Section 2.4]{Corni}. We utilize the estimate \eqref{ul} in \eqref{ulin} to obtain
	\begin{align}\label{w1}
		\qquad\qquad\int_0^te^{-(t-\tau)(-\Delta_{\mathbbm{H}^N})^s}&(|u(\tau)|^p+f)d\tau\\
		&\quad\leq \int_0^t\int_{\mathbbm{H}^N}h_{s}(t-\tau,\eta^{-1}\circ\xi)\left(\frac{CM^p}{1+d_{\mathbbm{H}^N}(\eta)^{Q+\delta}}+\frac{\varepsilon}{1+d_{\mathbbm{H}^N}(\eta)^{Q+\delta}}\right)d\eta d\tau\\
		&\quad\leq \int_0^\infty\int_{\mathbbm{H}^N}h_{s}(t-\tau,\eta^{-1}\circ\xi)\left(\frac{CM^p}{1+d_{\mathbbm{H}^N}(\eta)^{Q+\delta}}+\frac{\varepsilon}{1+d_{\mathbbm{H}^N}(\eta)^{Q+\delta}}\right)d\eta d\tau\\
		&\quad\leq \int_{\mathbbm{H}^N}\frac{C}{d_{\mathbbm{H}^N}(\eta,\xi)^{Q-2s}}\left(\frac{{C}M^p}{1+d_{\mathbbm{H}^N}(\eta)^{Q+\delta}}+\frac{\varepsilon}{1+d_{\mathbbm{H}^N}(\eta)^{Q+\delta}}\right)d\eta\\
		&\quad=\int_{\mathbbm{H}^N}\frac{C}{d_{\mathbbm{H}^N}(\eta,\xi)^{Q-2s}}\left(\frac{CM^p+\varepsilon}{1+d_{\mathbbm{H}^N}(\eta)^{Q+\delta}}\right)d\eta\\
		&\quad\leq \int_{\mathbbm{H}^N}\frac{C(CM^p+\varepsilon)}{d_{\mathbbm{H}^N}(\eta,\xi)^{Q-2s}(1+d_{\mathbbm{H}^N}(\eta)^{Q+\delta})}d\eta\\
		&\quad\leq \frac{C(CM^p+\varepsilon)}{(1+d_{\mathbbm{H}^N}(\xi))^{Q-2s}},
	\end{align}
	for some positive constant $C>0.$ The last inequality in \eqref{w1} follows by Lemma \ref{Lem Q2}. Finally, we choose $M$ and $\varepsilon$ so that $\Phi_M$ is well defined on $\Theta_M^s.$   

Next, we aim to show that $\Phi_M$ is a contraction. Note that , in view of \eqref{delt}, we have
\begin{align}
	(Q-2s)(p-1)&=(Q-2s)p-Q+2s\\
	&>Q+\delta-Q+2s\\&=2s+\delta,
\end{align}
for some $\delta>0.$
Making use of \eqref{d1} and \eqref{ul}, we deduce
\begin{align}
	\|\Phi_M(u_1)(\xi)-\Phi_M(u_2)(\xi)\|_{{\Theta_M^s}}&=\bigg\|\int_0^te^{-(t-\tau)(-\Delta_{\mathbbm{H}^N})^s}\big(|u_1(\tau)|^p-|u_2(\tau)|^p\big)d\tau\bigg\|_{\Theta_M^s}\\
	&=\bigg\|\int_0^t\int_{\mathbbm{H}^N}h_s(\tau,\eta^{-1}\circ\xi)\big(|u_1(\tau)|^p-|u_2(\tau)|^p\big)(\eta)d\tau d\eta\bigg\|_{\Theta_M^s}\\
	&\leq 2^{p-1}C(p)\|u_1-u_2\|_{L^\infty(\mathbbm{H}^N)}\int_0^\infty\int_{\mathbbm{H}^N}h_{s}(t-\tau,\eta^{-1}\circ\xi)\frac{{M^{p-1}}}{{(1+d_{\mathbbm{H}^N}(\eta)^{Q-2s}})^{p-1}} d\eta\\
	&\leq 2^{p-1}C(p)\|u_1-u_2\|_{L^\infty(\mathbbm{H}^N)}\int_{\mathbbm{H}^N}\frac{{M^{p-1}}}{({1+d_{\mathbbm{H}^N}(\eta)^{Q-2s}})^{p-1}d_{\mathbbm{H}^N}(\eta,\xi)^{Q-2s}} d\eta\\
	&\leq 2^{p-1}C(p)\|u_1-u_2\|_{L^\infty(\mathbbm{H}^N)}\int_{\mathbbm{H}^N}\frac{{M^{p-1}}}{({1+d_{\mathbbm{H}^N}(\eta)^{(Q-2s){(p-1)}}})d_{\mathbbm{H}^N}(\eta,\xi)^{Q-2s}} d\eta\\
	&\leq 2^{p-1}C(p){M^{p-1}}\|u_1-u_2\|_{L^\infty(\mathbbm{H}^N)}\int_{\mathbbm{H}^N}\frac{1}{({1+d_{\mathbbm{H}^N}(\eta)^{2s+\delta}})d_{\mathbbm{H}^N}(\eta,\xi)^{Q-2s}} d\eta\\
	&\leq 2^{p-1}C_2C(p){M^{p-1}}\|u_1-u_2\|_{L^\infty(\mathbbm{H}^N)},
\end{align}
where the last inequality follows by Lemma \ref{Lem Q5}. This shows that $\Phi_M$ is a contraction from $\Theta_M^s$ to itself for suitably small $M.$ The Banach fixed theorem confirms the existence of a mild solution to \eqref{eq 0.1}. This concludes the proof.		
\qed\\

Subsequently, we prove the non-existence part, i.e., Theorems \ref{Thm Glo1} and \ref{Thm Glo3}.\\

\noindent \textbf{Proof of Theorem \ref{Thm Glo1}.}
We give a proof by the method of contradiction. Let, if it exists, $u$ be a global-in-time solution to \eqref{eq 0.1}. For a fixed \begin{align}\label{lle}
    l=1+\frac{p}{p-1},
\end{align}
we define a function $$\phi(\xi,t)=(\phi_1(\xi))^l(\phi_2(t))^l.$$ Here
\begin{align}
	\phi_1(\xi)=\varphi\left(\frac{\|\xi\|_{\mathbbm{H}^N}}{T^{\frac{1}{2}}}\right), \, \phi_2(t)=\varphi\left(\frac{t}{T^{s}}\right)
\end{align}
and $\varphi:\mathbb{R}^+\cup\{0\}\To\mathbb{R}$ is a smooth function defined as
\begin{align}
	\varphi(r)=\begin{cases}
		1 &\text{ for } r\in [0,1],\\
		\searrow &\text{ for } r\in [1,2],\\
		0  &\text{ for } r\in [2,\infty).
	\end{cases}
\end{align}
Now, by Definition \ref{Weak}, we have
\begin{align}\label{am}
	\int_0^\infty\int_{\mathbbm{H}^N}|u|^p\phi +\int_{\mathbbm{H}^N}u_0\phi(0,\xi)+\int_0^\infty\int_{\mathbbm{H}^N} f\phi&=-\int_0^\infty\int_{\mathbbm{H}^N}u(\partial_t\phi-(-\Delta_{\mathbbm{H}^N})^s\phi)\\
	&=-\int_{S}\int_{\mathbbm{H}^N}u(\partial_t\phi-(-\Delta_{\mathbbm{H}^N})^s\phi),
\end{align}
where $$S\coloneqq [0,2T^{s}).$$ As an application of \cite[Proposition 2.1]{Ahmad}, we have that 
\begin{align}\label{km}
	(-\Delta_{\mathbbm{H}^N})^s(\phi_1(\xi))^l\leq l(\phi_1(\xi))^{l-1}(-\Delta_{\mathbbm{H}^N})^s\phi_1(\xi).
\end{align}
Consider
\begin{align}
	B^{\mathbbm{H}^N}_{2\sqrt T}\coloneqq \{\xi\in\mathbbm{H}^N: \|\xi\|_{\mathbbm{H}^N}\leq 2\sqrt T\}.
\end{align}
This implies, in light of \eqref{am}, that
\begin{align}\label{ka}
	\qquad\quad\int_S\int_{B^{\mathbbm{H}^N}_{2\sqrt T}}|u|^p\phi +\int_{B^{\mathbbm{H}^N}_{2\sqrt T}}u_0\phi(0,\cdot)+\int_S\int_{B^{\mathbbm{H}^N}_{2\sqrt T}} f\phi
	&=\int_0^\infty\int_{\mathbbm{H}^N}|u|^p\phi +\int_{\mathbbm{H}^N}u_0\phi(0,\cdot)+\int_0^\infty\int_{\mathbbm{H}^N} f\phi\\
	&= -\int_{S}\int_{\mathbbm{H}^N}u\partial_t\phi+\int_S\int_{\mathbbm{H}^N}u(-\Delta_{\mathbbm{H}^N})^s\phi\\
	&\leq -\int_{S}\int_{\mathbbm{H}^N}|u|\partial_t\phi+\int_S\int_{\mathbbm{H}^N}|u||(-\Delta_{\mathbbm{H}^N})^s\phi|\\
	&=-\int_{S}\int_{B^{\mathbbm{H}^N}_{2\sqrt T}}|u(\xi,t)|(\phi_1(\xi))^l\partial_t(\phi_2(t))^ld\xi dt\\
	&\qquad+\int_S\int_{\mathbbm{H}^N}|u(\xi,t)|(\phi_2(t))^l|(-\Delta_{\mathbbm{H}^N})^s(\phi_1(\xi))^l|d\xi dt\\
	&=-l\int_{S}\int_{B^{\mathbbm{H}^N}_{2\sqrt T}}|u(\xi,t)|(\phi_1(\xi))^l(\phi_2(t))^{l-1}\partial_t\phi_2(t)d\xi dt\\
	&\qquad+\int_S\int_{\mathbbm{H}^N}|u(\xi,t)|(\phi_2(t))^l|(-\Delta_{\mathbbm{H}^N})^s(\phi_1(\xi))^l|d\xi dt.
\end{align}
Let us consider
\begin{align}
	a_1=|u(\xi,t)|\phi_1(\xi)^{\frac{l}{p}}\phi_2(t)^{\frac{l}{p}},\,\,
	b_1=\phi_1(\xi)^{\frac{l}{l-1}}\phi_2(t)^{\frac{1}{l-1}}|\partial_t\phi_2(t)|,
\end{align}
and
\begin{align}
	a_2=|u(\xi,t)|\phi_1(\xi)^{\frac{l}{p}}\phi_2(t)^{\frac{l}{p}},\,\,
	b_2=\phi_1(\xi)^{\frac{1}{l-1}}\phi_2(t)^{\frac{l}{l-1}}|(-\Delta_{\mathbbm{H}^N})^s\phi_1(\xi)|,
\end{align}
where $l$ is given by \eqref{lle}.
Next, for some $\varepsilon>0,$ using the $\varepsilon$-Young inequality, we have 
\begin{align}\label{varyou}
	a_ib_i\leq \varepsilon a_i^p+C(\varepsilon)b_i^{\frac{p}{p-1}}, \text{ for }i=1,2. 
\end{align}
Utilizing \eqref{varyou}, \eqref{km} in \eqref{ka}, we obtain in view of the fact $\phi_1\in C_c^\infty(\mathbbm{H}^N)$ that
\begin{align}
	\int_S\int_{B^{\mathbbm{H}^N}_{2\sqrt T}}|u|^p\phi &+\int_{B^{\mathbbm{H}^N}_{2\sqrt T}}u_0\phi(0,\cdot)+\int_S\int_{B^{\mathbbm{H}^N}_{2\sqrt T}} f\phi\\
	&\leq l\varepsilon\int_{S}\int_{B^{\mathbbm{H}^N}_{2\sqrt T}}|u(\xi,t)|^p\phi_1(\xi)^{l}\phi_2(t)^{l}d\xi dt+lC(\varepsilon)\int_{S}\int_{B^{\mathbbm{H}^N}_{2\sqrt T}}\phi_1(\xi)^l\phi_2(t)|\partial_t\phi_2(t)|^{l-1}d\xi dt\\
	&\qquad+l\varepsilon\int_{S}\int_{B^{\mathbbm{H}^N}_{2\sqrt T}}|u(\xi,t)|^p\phi_1(\xi)^{l}\phi_2(t)^{l}d\xi dt+lC_1C(\varepsilon)\int_S\int_{\mathbbm{H}^N}\phi_2(t)^l\phi_1(\xi)|(-\Delta_{{\mathbbm{H}^N}})^s\phi_1(\xi)|^{l-1}d\xi dt\\
	&= l\varepsilon\int_{S}\int_{B^{\mathbbm{H}^N}_{2\sqrt T}}|u(\xi,t)|^p\phi(\xi,t)d\xi dt+lC(\varepsilon)\int_{S}\int_{B^{\mathbbm{H}^N}_{2\sqrt T}}\phi_1(\xi)^l\phi_2(t)|\partial_t\phi_2(t)|^{l-1}d\xi dt\\
	&\qquad+l\varepsilon\int_{S}\int_{B^{\mathbbm{H}^N}_{2\sqrt T}}|u(\xi,t)|^p\phi(\xi,t)d\xi dt+C_1lC(\varepsilon)\int_S\int_{B^{\mathbbm{H}^N}_{2\sqrt T}}\phi_2(t)^l\phi_1(\xi)|(-\Delta_{\mathbbm{H}^N})^s\phi_1(\xi)|^{l-1}d\xi dt,
\end{align}
for some constant $C_1>0$ independent of $T.$
The above inequality further simplifies to
\begin{align}\label{l1q}
	(1-2l\varepsilon)\int_S\int_{B^{\mathbbm{H}^N}_{2\sqrt T}}|u|^p\phi+\int_{B^{\mathbbm{H}^N}_{2\sqrt T}}&u_0\phi(0,\cdot)+\int_S\int_{B^{\mathbbm{H}^N}_{2\sqrt T}} f\phi\\
	&\qquad\leq lC(\varepsilon)\int_{S}\int_{B^{\mathbbm{H}^N}_{2\sqrt T}}\phi_1(\xi)^l\phi_2(t)|\partial_t\phi_2(t)|^{l-1}d\xi dt\\
	&\qquad\qquad+C_1lC(\varepsilon)\int_S\int_{B^{\mathbbm{H}^N}_{2\sqrt T}}\phi_2(t)^l\phi_1(\xi)|(-\Delta_{\mathbbm{H}^N})^s\phi_1(\xi)|^{l-1}d\xi dt.
\end{align}
Next, using the change of variables $$\xi'=\Phi_{T^{-\frac{1}{2}}}\xi \text{ and } t'=T^{-s}t,$$ in \eqref{l1q} together with the following well-known intrinsic homogeneity of $(-\Delta_{\mathbbm{H}^N})^s:$  
\begin{align}
	(-\Delta_{\mathbbm{H}^N})^s(u(\Phi_\lambda \xi))=\lambda^{2s}(-\Delta_{\mathbbm{H}^N})^su(\Phi_\lambda (\xi)),
\end{align}
deduces
\begin{align}\label{lep}
	\qquad(1-2l\varepsilon)\int_S\int_{B^{\mathbbm{H}^N}_{2\sqrt T}}|u|^p\phi&+\int_{B^{\mathbbm{H}^N}_{2\sqrt T}}u_0\phi(0,\cdot)+\int_S\int_{B^{\mathbbm{H}^N}_{2\sqrt T}} f\phi\\
	&\leq lC(\varepsilon)T^{\frac{Q}{2}}T^{s}T^{-s(l-1)}\int_0^2\int_{B^{\mathbbm{H}^N}_{2}}\phi_1(\xi')^l\phi_2(t')|\partial_t\phi_2(t')|^{l-1}d\xi'dt'\\
	&\qquad+C_1lC(\varepsilon)T^{\frac{Q}{2}}T^{s}T^{-s(l-1)}\int_0^2\int_{B^{\mathbbm{H}^N}_{2}}\phi_2(t')^l\phi_1(\xi')|(-\Delta_{\mathbbm{H}^N})^s\phi_1(\xi')|^{l-1}d\xi'dt'\\
	&=lC(\varepsilon)T^{\frac{Q}{2}+s-\frac{sp}{p-1}}\int_{0}^2\int_{B^{\mathbbm{H}^N}_{2}}\phi_1(\xi')^l\phi_2(t')|\partial_t\phi_2(t')|^{l-1}d\xi'dt'\\
	&\qquad+C_1lC(\varepsilon)T^{\frac{Q}{2}+s-\frac{sp}{p-1}}\int_0^2\int_{B^{\mathbbm{H}^N}_{2}}\phi_2(t')^l\phi_1(\xi')|(-\Delta_{\mathbbm{H}^N})^s\phi_1(\xi')|^{l-1}d\xi'dt'\\
	&\leq lC(\varepsilon)T^{\frac{Q}{2}+s-\frac{sp}{p-1}}\int_{0}^2\int_{B^{\mathbbm{H}^N}_{2}}\phi_1(\xi')^l\phi_2(t')|\partial_t\phi_2(t')|^{(l-1)}d\xi'dt'\\
	&\qquad+lC_2C(\varepsilon)T^{\frac{Q}{2}+s-\frac{sp}{p-1}}\int_0^2\int_{B^{\mathbbm{H}^N}_{2}}\phi_2(t')^l\phi_1(\xi')d\xi'dt',
\end{align}
for some constant $C_2$ independent of $T,$ since $\phi_1\in C_c^\infty(\mathbbm{H}^N).$ Using the facts that $u_0\geq 0$ and $\varepsilon$ is arbitrarily small together with the following relation
\begin{align}\label{s3}
	\int_S\int_{\mathbbm{H}^N}f\phi&=\int_{B^{\mathbbm{H}^N}_{2\sqrt T}}f\phi_1\int_S\phi_2\\
	&\geq 2CT^s\int_{\mathbbm{H}^N}f\phi_1
\end{align}
in the expression \eqref{lep} yields
\begin{align}
	2T^s\int_{\mathbbm{H}^N}f\phi_1&\leq lC(\varepsilon)T^{\frac{Q}{2}+s-\frac{sp}{p-1}}\int_{S}\int_{B^{\mathbbm{H}^N}_{2\sqrt T}}\phi_1(\xi')^l\phi_2(t')|\partial_t\phi_2(t')|^{l-1}d\xi' dt'\\
	&\qquad+C_2lC(\varepsilon)T^{\frac{Q}{2}+s-\frac{sp}{p-1}}\int_S\int_{B^{\mathbbm{H}^N}_{2\sqrt T}}\phi_2(t')^l\phi_1(\xi')d\xi'dt'\\
	&\leq CT^{\frac{Q}{2}+s-\frac{sp}{p-1}},
\end{align}
for some positive constant $C$ independent of $T.$ This simplifies to
\begin{align}\label{s1}
	\int_{\mathbbm{H}^N}f\phi_1&=\int_{B^{\mathbbm{H}^N}_{2\sqrt T}}f\phi_1\\
	&\leq CT^{\frac{Q}{2}+s-\frac{sp}{p-1}}T^{-s}\\
	&=CT^{\frac{Q}{2}-\frac{sp}{p-1}}
\end{align}
Furthermore, observe that 
\begin{align}\label{s2}
	\frac{Q}{2}-\frac{sp}{p-1}
	&=\frac{Q(p-1)-2sp}{2(p-1)}\\
	&=\frac{p(Q-2s)-Q}{2(p-1)}\\
	&< 0,
\end{align}
since $$1\leq p<\frac{Q}{Q-2s}.$$
Finally, in light of \eqref{s2}, taking the limit $T\to\infty$ in \eqref{s1} yields
\begin{align}
	\int_{\mathbbm{H}^N}f\leq 0,
\end{align}
which is a contradiction to our hypothesis that $\int_{\mathbbm{H}^N}f>0.$ This proves our claim.\qed\\

\noindent \textbf{Proof of Theorem \ref{Thm Glo3}.}
We give a proof by the method of contradiction. Let, if possible, $u$ be a global-in-time solution to \eqref{eq 0.1}. We recall from the proof of Theorem \ref{Thm Glo1} that assumption $p<\frac{Q}{Q-2s}$ is used after estimate \eqref{s1}. In view of this, taking $p=\frac{Q}{Q-2s}$ in \eqref{s1} and \eqref{s2} deduces
\begin{align}
	\int_{\mathbbm{H}^N}f\phi_1&=\int_{B^{\mathbbm{H}^N}_{2\sqrt T}}f\phi_1\\
	&\leq C,
\end{align}
where $C$ is the same constant appearing in \eqref{s1}. Further, using the assumption that  
\begin{align}
	f(\xi)\geq \|\xi\|^{\alpha-Q}_{\mathbbm{H}^N} \text{ for } \|\xi\|_{\mathbbm{H}^N}\geq 1 \text{ and }\alpha\in(0,Q),
\end{align}
provides
\begin{align}
	\frac{T^{\frac{\alpha}{2}}}{\alpha}&=\int_{0}^{\sqrt T}r^{\alpha-1}dr\\
	&= \int_{0}^{\sqrt T}r^{\alpha-Q}r^{Q-1}dr\\
	&=C_1\int_{B^{\mathbbm{H}^N}_{\sqrt T}}\|\xi\|_{\mathbbm{H}^N}^{\alpha-Q}d\xi\\
	&=C_1\int_{B^{\mathbbm{H}^N}_{\sqrt T}\setminus B^{\mathbbm{H}^N}_{1}}\|\xi\|_{\mathbbm{H}^N}^{\alpha-Q}d\xi+C_1\int_{B^{\mathbbm{H}^N}_{1}}\|\xi\|_{\mathbbm{H}^N}^{\alpha-Q}d\xi\\
    &\leq C_1\int_{B^{\mathbbm{H}^N}_{\sqrt T}\setminus B^{\mathbbm{H}^N}_{1}}f+C_2\\
	&\leq C,
\end{align}
for some positive constants $C_1,C_2,C$ independent of $T.$ This further implies
\begin{align}
	1&\leq C\alpha T^{-\frac{\alpha}{2}}\\
	&< CQ T^{-\frac{\alpha}{2}},
\end{align}
which gives a contradiction on taking $T\To\infty$ as $C$ is independent of $T$ and $\alpha>0.$ This completes the proof. \qed
\stoptoc
\section{Declarations}


\subsection*{Funding } This research was funded by Nazarbayev University under Collaborative Research Program Grant 20122022CRP1601.

\subsection*{Availability of data and materials}     
Not Applicable.

\subsection*{Conflicts of interests/Competing interests}
There are no conflict of interest of any type.
\resumetoc

\end{document}